\documentclass{amsart}

\usepackage[english]{babel}
\usepackage{amsfonts, amsmath, amsthm, amssymb,amscd,indentfirst, mathtools}

\usepackage{mathrsfs}

\usepackage{hyperref}

\DeclareMathAlphabet{\mathpzc}{OT1}{pzc}{m}{it}

\usepackage{tikz-cd}

\usepackage{scalerel}
\usepackage{stackengine,wasysym}

\usepackage{color}
\usepackage{xcolor}
\usepackage{tikz}
\usepackage{esint}

\usepackage{esint}
\usepackage{graphicx}
\usepackage{epstopdf}
\usepackage{amsmath,amssymb,latexsym,indentfirst}
\usepackage{times}
\usepackage{palatino}

\usepackage{stackengine}
\usepackage{scalerel}

\newcommand{\iprod}{\mathbin{\lrcorner}}

\def\ba{\begin{align}}
	\def\ea{\end{align}}
\def\bp{\begin{proof}}
	\def\ep{\end{proof}}

\theoremstyle{plain}
\newtheorem{theorem}{Theorem}[section]

\newtheorem{proposition}[theorem]{Proposition}
\newtheorem{lemma}[theorem]{Lemma}
\newtheorem{corollary}[theorem]{Corollary}
\theoremstyle{definition}
\newtheorem{definition}[theorem]{Definition}

\newtheorem{example}[theorem]{Example}

\newtheorem{remark}[theorem]{Remark}

	\def\bp{\begin{proof}}
		\def\ep{\end{proof}}

	\def\Ac{{\mathcal A}}

\begin{document}

		\title[Einstein-Yang-Mills fields]{Einstein-Yang-Mills fields in conformally compact manifolds}

		\author{Levi Lopes de Lima}
		\address{Universidade Federal do Cear\'a (UFC),
			Departamento de Matem\'{a}tica, Campus do Pici, Av. Humberto Monte, s/n, Bloco
			914, 60455-760,
			Fortaleza, CE, Brazil.}
		\email{levi@mat.ufc.br}
		\thanks{
			Supported by 
			FUNCAP/CNPq/PRONEX 00068.01.00/15.}

		\begin{abstract}
			We study the deformation theory of Einstein-Yang-Mills fields over conformally compact, asymptotically locally hyperbolic manifolds. We prove that if an Einstein-Yang-Mills field $(g_0,\omega_0)$ is trivial (which means that $g_0$ is a Poincar\'e-Einstein metric and $\omega_0$ is a flat connection on a principal bundle over the underlying manifold) and non-degenerate in the appropriate sense then any sufficiently small perturbation of its boundary data at infinity may be realized as the boundary data of some Einstein-Yang-Mills field. This result is obtained as an application of the $0$-calculus of Mazzeo and Melrose \cite{mazzeo1987meromorphic,mazzeo1988hodge} and may be viewed as a natural extension of previous results by Graham-Lee \cite{graham1991einstein}, Lee \cite{lee2006fredholm} and Usula \cite{usula2021yang}.   
		\end{abstract}

		\maketitle


		\section{Introduction}\label{intro}
		
		Let $\overline M^{n+1}$, $n\geq 2$, be a compact manifold with boundary $Y=\overline M\backslash M$, where $M$ is its interior and let $g$ be a complete metric in $M$. We say that $(M,g)$ is {\em conformally compact} if there exists a collar neighborhood $\mathcal U$ of $Y$ in which a function $\mathfrak r$ is defined so that $\mathfrak r\geq 0$, $\mathfrak r^{-1}(0)=Y$ and  $|d\mathfrak r|_{\overline g}\neq 0$ along $Y$,
		where 
		$
		\mathfrak r^2g{\big|}_{\mathcal U}=:\overline g
		$ extends to a sufficiently regular metric up to $Y$.
		We then say that $\mathfrak r$ is a {\em defining function} (for $Y$). 
		Notice that the conformal class $\partial_\infty g:=[\mathfrak r^2g|_{Y}]$ does {not} depend on which defining function is used.
		We then say that
		the conformal manifold $(Y,\partial_\infty g)$ is the {\em conformal boundary} of $(M,g)$. Alternatively, $(M,g)$ is a {\em filling} of $(Y,\partial_\infty g)$.

		One verifies that 
		$K_{\rm sec}(g)=-|d\mathfrak r|^2_{\overline g}+o(1)$ as $ \mathfrak r\to 0$, which naturally leads to appropriate refinements. In the following we represent by $\mathfrak h_0$ the round metric on the unit $n$-sphere $\mathbb S^n$.
		
		\begin{definition}\label{alh:def}
			A conformally compact manifold $(M,g)$ is {\em asymptotically locally hyperbolic} (ALH) if $|d\mathfrak r|_{\overline g}=1$ along $Y$. Equivalently, $K_{\rm sec}(g)\to -1$ as $\mathfrak r\to 0$. If, moreover, $\partial_\infty g=[\mathfrak h_0]$, the conformal class of $\mathfrak h_0$, then we say that $(M,g)$ is {\em asymptotically hyperbolic} (AH). 
		\end{definition}
		
		An interesting class of ALH manifolds appears when we restrict the interior geometry appropriately.
		
		\begin{definition}\label{pe:def}
			A ALH manifold is {\em Poincar\'e-Einstein} (PE) if $g$ is Einstein. In other words, there holds $\prescript{g}{}{\rm Ric}=-ng$ identically. 
		\end{definition}
		
		The prototypical example of a PE manifold is the hyperbolic $(n+1)$-space $(\mathbb H^{n+1},g_{\mathbb H})$. In the Poincar\'e ball model, $\mathbb H^{n+1}=\{w\in\mathbb R^{n+1}; |w|<1\}$ and $g_{\mathbb H}=\mathfrak r^{-2}dw^2$, $\mathfrak r(w)=(1-|w|^2)/2$. Hence, $\partial_{\infty}g_{\mathbb H}=[\mathfrak h_0]$, so that $(\mathbb H^{n+1},g_{\mathbb H})$ is AH as well.
		
		PE manifolds play a central role both in the Fefferman-Graham theory of conformal
		invariants \cite{fefferman2012ambient} and in the AdS/CFT correspondence \cite{witten1998antide,biquard2005ads}. In this regard, it has been proved that 
		a AH and PE manifold is isometric to $(\mathbb H^{n+1},g_{\mathbb H})$ \cite{andersson1998scalar,qing2003rigidity,li2017gap}.  
		This {\em global} rigidity of $g_{\mathbb H}$ clearly suggests the use of perturbative methods to probe the structure of nearby PE metrics on the unit ball parameterized by small deformations of $[\mathfrak h_0]$. In fact, this expectation had been previously confirmed by the following fundamental result. 
		
		\begin{theorem}[Graham-Lee \cite{graham1991einstein}]\label{gl:dich}
			Any conformal class sufficiently close to $[\mathfrak h_0]$ admits a {unique} PE filling. 
		\end{theorem}
		
		Uniqueness here means up to the action of the gauge group of diffeomorphisms of the ball fixing the boundary. Further contributions in this direction, where $g_{\mathbb H}$ is replaced by more general background metrics under a suitable non-degeneracy assumption, may be found in \cite{lee2006fredholm,biquard2000metriques,biquard2006parabolic,mazzeo2006maskit,biquard2011nonlinear,bahuaud2020geometrically,fine2022non}.
		
		Our aim here is to properly extend the validity of the nonlinear Dirichlet problem at infinity in Theorem \ref{gl:dich} to the Einstein-Yang-Mills setting, again by perturbative methods. 
		More precisely, and inspired by \cite{usula2021yang}, we consider a configuration $(g,\omega)$ formed by a metric $g$ over $M$ as above (in particular, $(M,g)$ is assumed to be conformally compact) and  a connection $\omega$ on a principal bundle over $M$. As in \cite{bleecker2005gauge} we set up a variational problem on the space of all such configurations whose critical points, the so-called {\em Einstein-Yang-Mills fields}, satisfy a coupled system of equations involving their curvature invariants (the Ricci tensor of $g$ and the curvature $2$-form of $\omega$).  Trivial solutions  of this system are obtained if $g$ is PE and $\omega$ is flat (in the sense that its curvature vanishes identically). Roughly, our main result says that whenever a trivial solution, say $(g_0,\omega_0)$, is {\em non-degenerate} in a suitable sense then any sufficiently small perturbation of $(\partial_\infty g_0,\omega_0\big|_{Y})$ may be realized as the boundary data at infinity of a Einstein-Yang-Mills field defined on $M$, which is unique up to the action of the group of gauge transformations fixing the boundary data. We refer to Theorem \ref{main} for a precise statement. 
		
		In our setting, non-degeneracy means that both $g_0$ and $\omega_0$ are separately non-degenerate in the sense of \cite{graham1991einstein} and \cite{usula2021yang}, respectively; see Definition \ref{nondeg:conf}. We remark that the notion of non-degeneracy in \cite{usula2021yang} is more general than ours as it does not require flatness for  $\omega_0$; see Remarks  \ref{inc:discuss} and \ref{non:flat}. Our more specialized definition reflects the fact explored here that, as shown in Lemma \ref{differ:cal}, the linearization of  the Einstein-Yang-Mills field equations at a trivial configuration conveniently decouples as a sum of its metric and connection linearizations (this relies on the fact that the corresponding stress-energy tensor, the metric variation of the Yang-Mills self-action density, depends {\em quadratically} on the curvature of the connection; see Remark \ref{K:quad}). 
		In particular, if we take $(M,g_0)$ to be a non-degenerate PE manifold, $\omega_0$ the canonical flat connection on the trivial principal bundle $M\times\mathbb S^1$ and restrict ourselves to perturbations only in the metric direction then Theorem \ref{main} recovers \cite[Theorem A]{lee2006fredholm}, where a  large class of PE manifolds was shown to be non-degenerate; also note that \cite{fine2022non} exhibits a new class of non-degenerate PE metrics in dimension four. On the other hand, if we fix a PE manifold $(M,g_0)$ satisfying $H^1(\overline M,Y)=\{0\}$ then the canonical flat connection $\omega_0$ on a trivial principal bundle over $M$ is known to be non-degenerate, so if we vary only in the direction of the connection then Theorem \ref{main} recovers \cite[Corollary 39]{usula2021yang}.  We insist, however, that for more general deformations Theorem \ref{main} certainly produces genuinely coupled Einstein-Yang-Mills fields around a trivial solution. In this regard, we refer to Theorem \ref{witten}, which is our version of  Usula's result just mentioned. 
		
		As expected from past experience, when trying to implement the appropriate perturbative scheme in our setting, two complications immediately arise. The first one relates to the existence of a large gauge symmetry group for the field equations, which in particular prevents them from being elliptic. To overcome this we combine the schemes in \cite{biquard2000metriques,mazzeo2006maskit,usula2021yang} so that by imposing the {\em Bianchi-Coulomb gauge} we are able to construct a local slice for the gauge action along which ellipticity (i.e. pointwise invertibility of the principal symbol) of the linearized field equations is restored. Next, since this gauged linearized operator acts on sections of vector bundles over a conformally compact manifold, ellipticity alone does not suffice to ensure its Fredholmness, much less its invertibility, in a suitable weighted H\"older scale. 
		The key  point here is that the principal symbol of this linearized operator is essentially determined by the contravariant metric, which vanishes along the conformal boundary; see Example \ref{lap:exp:loc}, where this phenomenon is explicitly described for the scalar Laplacian. 
		Thus, interior ellipticity should somehow be complemented with a notion of ellipticity along the conformal boundary, which involves the invertibiliby of certain ``model'' operators. There are at least two ways to overcome this difficulty. First, we may adapt the ``low tech'' methods in \cite{graham1991einstein,andersson1993elliptic,andersson1996solutions,lee2006fredholm}, which for instance have been systematically employed to establish the existence of static/stationary space-time negative cosmological constant solutions of Einstein-Yang-Mills equations, possibly of black hole type and coupled with matter fields (see the series of papers culminating in \cite{chrusciel2018non}). Second, we may simply apply the powerful $0$-calculus for uniformly degenerate elliptic operators put forward in \cite{mazzeo1988hodge,mazzeo1987meromorphic}, whose variants have been extensively used over the years, notably in the deformation theory of Einstein metrics; see \cite{biquard2006parabolic,biquard2011nonlinear,mazzeo2006maskit,bahuaud2020geometrically,usula2021yang,fine2022non} for a sample of such applications. We point out that yet another route to the analytical machinery used here derives from the pseudo-differential calculus on manifolds with a ``Lie structure at infinity'' \cite{ammann2004geometry,ammann2007pseudodifferential,carvalho2018fredholm}. In this work 
		we have chosen to frame our approach in the $0$-calculus setting not only because it seems best suited for further developments but also because its implementation naturally leads to the $L^2$ non-degeneracy requirement which is so ubiquitous in applications (see Remark \ref{non-deg:req}).      
		
		This paper is organized as follows. In Section \ref{fred:0-calc} we recall the basics of the $0$-calculus as applied to generalized Laplacians. This is used in Section \ref{inv:oper} to establish the mapping properties (Fredholmness and invertibility) of certain Laplace-type operators arising in the gauge theory of Einstein-Yang-Mills fields considered in Section \ref{eym:disc}. After the Bianchi-Coulomb gauge  fixing is carried out in Section \ref{fix:gauge}, we present the proof of our main result (Theorem \ref{main}) in Section \ref{main:res}. We also include a final Section \ref{further}, where possible extensions of our main theorem are briefly discussed.    
		
		
		\section{Fredholmness for generalized Laplacians in the $0$-calculus}
		\label{fred:0-calc}
		
		Roughly speaking, the $0$-calculus has been designed to handle the mapping properties of {\em uniformly degenerate} operators in ALH manifolds. That this kind of degeneracy certainly occurs for {\em geometric} differential operators should cause no surprise as we have seen that only the conformal structure survives as $\mathfrak r\to 0$. 
		The purpose of this section is precisely to describe how the general theory in \cite{mazzeo1987meromorphic,mazzeo1988hodge} may be applied to a class of second order elliptic operators appearing frequently  in geometric applications. Expositions of this material at a foundational level may be found in \cite{lauter2003pseudodifferential,albinnotes2008,hintz2021elliptic}. More streamlined presentations, decorated with interesting applications to Geometry, appear in \cite{mazzeo2006maskit,usula2021yang,fine2022non}. Needless to say, our account here owes a lot to these sources.    
		
		As a first step towards our goal we must
		understand how geometric objects degenerate as we approach the conformal boundary. For this we must properly choose the defining function, which is accomplished by the following well-known result.
		
		\begin{proposition}
			Let $(M^{n+1},g)$ be a ALH manifold with conformal boundary $(Y,\partial_\infty g)$. Then for any metric $h_0\in\partial_\infty g$ there exists a collar neighborhood $\mathcal U$ of $Y$ and a {\em special} defining function $x:\mathcal U\to \mathbb R^+$ so that 
			\begin{equation}\label{met:spec}
				x^2g|_{\mathcal U}=\overline g, \quad \overline g=dx^2+h(x),
			\end{equation}
			where $h(x)$ is a one-parameter smooth family of metrics in $Y$ satisfying $h(0)=h_0$. 
			In particular, $|dx|_{\overline g}=1$ on $\mathcal U$.
		\end{proposition}

		With a special defining function $x$ as above at hand,
		the general strategy in the $0$-category is to somehow ``desingularize'' geometric objects (vector fields, differential forms, etc.) by viewing them as ``smooth'' sections of certain vector bundles defined up to $Y$ which are designed to replace their classical counterparts, with the added bonus that the corresponding bundles remain isomorphic in the interior. More precisely, near infinity we may introduce coordinates $z=(x,y)$, where $y=(y_1,\dots,y_n)$ are local coordinates along $Y$. For convenience we assume that $x$ has been extended to the whole of $\overline M$ so that it takes  a {\em constant} positive value outside $\mathcal U$. We may start by  considering $\mathcal X_0(X)$, the space of $0$-vector fields, as being formed by those elements of $\mathcal X(\overline M)$, the Lie algebra of smooth vector fields on $\overline M$, vanishing along $Y$. Thus, an application of Serre-Swann theorem leads to the identification $\mathcal X_0(M)=C^{\infty}(\prescript{0}{}{TM})$, the space of smooth sections of $\prescript{0}{}{TM}$, the $0$-tangent bundle of $M$. Locally near some $z\in Y$, 
		\[
		\mathcal X_0(M)={\langle x\partial_x}, {x\partial_{y_1}, \dots, x\partial_{y_n}}\rangle_{C^\infty(\overline M)}.
		\]
		By dualizing this we obtain, again near infinity,
		\[
		\mathcal A^1_0(M)={\langle{dx}/{x}},{{dy_1}/{x},\dots,{dy_n}/{x}}\rangle_{C^\infty(\overline M)},
		\]
		so that $\mathcal A^1_0(M)=C^\infty(\prescript{0}{}{T^*M})$, the space of $0$-$1$-forms. In general, given an integer $k\geq 0$ we denote by $S^k$ and $\Lambda^k$ the standard operations of symmetrization and anti-symmetrizations, respectively. Thus,  $\mathcal A^k_0(M)=C^\infty({\Lambda}^k\prescript{0}{}{T^*M})$, the space of $0$-$k$-forms. 
		Also, if $\overline{E}$ is a vector bundle over $\overline M$ (whose restriction to $M$ we will always denote by its unbarred version $E$) we may consider $\mathcal A^k_0(M,\overline E)=C^\infty({\Lambda}^k\prescript{0}{}{T^*M}\otimes {\overline{E}})$, the space of $0$-$k$-forms with values in $\overline{E}$. 
		Finally, we may also consider $\mathcal S_0^k(M)=C^\infty(S^k{\prescript{0}{}{T^*M}})$. 
		It then follows from (\ref{met:spec}) that $g$ defines a positive definite element in $\mathcal S_0^2(M)$, a $0$-metric. Roughly, this means that $g$ extends to $\overline M$ with a ``double pole'' along $Y$.

		We now single out the class of elliptic operators we consider here. 
		
		\begin{definition}\label{lap:type}
			Let $(M,g)$ be a (not necessarily ALH) Riemannian manifold and let $ E$ be a metric vector bundle over $M$.   We say that a linear operator $\mathcal L:C^\infty(E)\to C^\infty(E)$ is a {\em generalized Laplacian} if its 
			principal symbol is
			\begin{equation}\label{symb:ell}
				{\rm Symb}_{\mathcal L}(p,\xi)=|\xi|^2_{ g}= g^{ij}(p)\xi_i\xi_j, \quad (p,\xi)\in T^* M.
			\end{equation}
		\end{definition}
		
		In particular, any generalized Laplacian is elliptic. 
		
		\begin{proposition}\label{lap:type:2}\cite[Proposition 2.5]{berline2003heat}
			If $\mathcal L:C^\infty(E)\to C^\infty(E)$ is a {generalized Laplacian} as above then there exists a unique metric connection $\nabla:C^\infty( E)\to C^\infty(T^* M\otimes E)$ and a zero order operator $\mathcal R\in C^\infty({{\rm End}\, E})$, the Weitzenb\"ock potential, such that  $\mathcal L=\nabla^*\nabla+\mathcal R$, where 
			$\nabla^*:C^\infty(T^* M\otimes E)\to C^\infty( E)$ is the formal adjoint of $\nabla$, so that $\nabla^*\nabla$ is the corresponding {\em Bochner Laplacian}.
		\end{proposition}
		
		\begin{definition}\label{geom:op}
			A generalized Laplacian $\mathcal L$ is {\em geometric} if:
			\begin{itemize}
				\item 
				the metric bundle $E$ and the corresponding compatible connection $\nabla$ (as in the previous proposition) are associated to the principal orthonormal frame bundle of $( M, g)$ via some representation of the orthogonal group;
				\item $\mathcal R$ is a (fiberwise) self-adjoint endomorphism constructed out of the curvature tensor of $\nabla$ (by means of contractions).  
			\end{itemize} 
		\end{definition}
		
		Examples of geometric Laplacians include the Hodge Laplacian (acting on differential forms) and the Lichnerowicz Laplacian (acting on symmetric $2$-covariant tensors).

		In general, if $E$ and $F$ are vector bundles over a ALH manifold $(M,g)$ and $P:C^\infty(E)\to C^\infty(F)$ is a differential operator then the mapping properties of $P$ (Fredholmness, invertibility, etc.) are amenable of being treated by the $0$-calculus if 
		near infinity $P$ may be expressed 
		as a (matrix of) polynomial expressions in elements of $\mathcal X_0(M)$ with coefficients in $C^\infty(\overline M)$. This class of operators, which we collectively represent by ${\rm Diff}_0^\bullet(\overline E,\overline F)$, where the bullet indicates the existence of a natural grading, are called $0$-{\em differential operators}. 
		This is certainly the case of a generalized Laplacian $\mathcal L$, for which there holds  
		\begin{equation}\label{gen:lap:exp}
			\mathcal L\big|_{\mathcal U}=\sum_{j+|\beta|\leq 2}a_{j,\beta}(x,y)(x\partial_x)^j(x\partial_y)^\beta.
		\end{equation}
		Thus, $\mathcal L\in {\rm Diff}_0^2(\overline E):={\rm Diff}_0^2(\overline E,\overline E)$. 
		
		\begin{remark}\label{rest:ext}
			Elements of ${\rm Diff}_0^\bullet(\overline E,\overline F)$ define differential operators in the interior (in particular, it makes sense to define an ``interior'' symbol and hence discuss ellipticity for such operators). 
			Also, ``genuine'' differential operators on $\overline M$ (i.e. elements of  ${\rm Diff}^\bullet(\overline E,\overline F)$) restrict to $0$-differential operators but the converse extension statement is not true in general. Thus, it is convenient to adopt here the convention that any $0$-differential operator which is not  geometric (e.g. the covariant derivative associated to a connection on an ``external'' principal bundle over $M$) is the restriction of a ``genuine'' object defined on $\overline M$.  As we shall see in the specific case of covariant derivatives, this amounts to requiring that the underlying connection is genuine and has the practical effect of turning the associated generalized Laplacians increasingly geometric as we approach the boundary; compare with \cite[Remark 22]{usula2021yang} and the proof of 
			Theorem \ref{usula:inv}.
		\end{remark}

		\begin{example}\label{lap:exp:loc}
			Starting with (\ref{met:spec}) one computes that the scalar Laplacian  satisfies
			\[
			\prescript{g}{}{\Delta}{\big|}_{\mathcal U} 
			=  -{x^2}(\partial_{xx}+\Delta_{h_0}+O(x)+(1-n)x^{-1}\partial_x),
			\]
			which makes the degeneracy of its principal symbol manifest as $x\to 0$ (have in mind that $g^{xx}=x^2+o(1)$). However, using that $x^{2}\partial_{xx}=(x\partial_x)^2-x\partial_x$ we get
			\begin{equation}\label{lap:ind}
				\prescript{g}{}\Delta{\big|}_{\mathcal U}=-(x\partial_x)^2+nx\partial_x-{x^2\Delta_{h_0}}+o(1),
			\end{equation}
			confirming that $\prescript{g}{}{\Delta}\in {\rm Diff}_0^2(\underline{\mathbb R})$, where $\underline{\mathbb R}$ is the trivial line bundle. Notice that our sign convention for the scalar Laplacian is such that it satisfies (\ref{symb:ell}) and hence is a generalized Laplacian.  
		\end{example}
		
		We now look at the functional spaces where the mapping properties of a generalized Laplacian should  be considered. Thus, if $(M,g)$ is an ALH space as above fix a collar neighborhood of $X$ on which a special defining function $x$ is available. As a reference functional space we take
		$C^{0,\alpha}_0$, the standard H\"older space of functions on $M$ endowed with the usual norm. 
		We then define $C_{0}^{k,\alpha}$ to be the Banach space of all functions $u$ such that 
		\[
		{(x\partial_x)^j}{(x\partial_y)^\beta}u\in
		C_0^{0,\alpha}, \quad j+|\beta|\leq k.
		\]
		Finally, if $\delta\in\mathbb R$ we set
		\[
		x^\delta C_{0}^{k,\alpha}=\left\{
		u: x^{-\delta}u\in 	C_{0}^{k,\alpha}
		\right\}.
		\]
		Similarly, we may define $x^\delta C_{0}^{k,\alpha}(\overline E)$, where $\overline E$ is a vector bundle over $\overline M$ (the previous case corresponds to taking $\overline E$ as the trivial line bundle).

		It is clear from (\ref{gen:lap:exp}) that any generalized Laplacian defines a bounded map
		\[
		\mathcal L:x^\delta C_{0}^{k,\alpha}(\overline E)\to x^\delta C_{0}^{k-2,\alpha}(\overline E), 
		\]
		for any $\delta$. The question remains of checking whether this map is invertible, or at least Fredholm, for some $\delta$. Simple examples show that interior ellipticity does not suffice to accomplish this. In fact,  a complementary notion of ellipticity along $Y$ is required which involves the consideration of a certain ``model operator''. 
		In terms of the representation (\ref{gen:lap:exp}), this {\em normal operator} is given by
		\begin{equation}\label{norm:rep:def}
			\mathcal N_{(x_0,y_0)}(\mathcal L)=\sum_{j+|\beta|\leq 2}a_{j,\beta}(0,0)(t\partial_t)^j(t\partial_\xi)^\beta.
		\end{equation}
		In general, it may be viewed as acting on sections of the trivial bundle $T^+_{(x_0,y_0)}\overline M\times \overline E_{(x_0,y_0)}$ over $T^+_{(x_0,y_0)}\overline M$, the inward pointing half-space of $T_{(x_0,y_0)}\overline M$, which is naturally endowed with coordinates $(t,\xi)$, $t\geq 0$, $\xi\in\mathbb R^n$, and the hyperbolic metric  $t^{-2}(dt^2+d\xi^2)$. 
		In particular, the normal operator defines a bounded map
		\[
		\mathcal N_{(x_0,y_0)}(\mathcal L):t^\delta H^{k}_0(\overline E_{(x_0,y_0)})\to t^\delta H^{k-2}_0(\overline E_{(x_0,y_0)}),
		\]		
		where the weighted $0$-Sobolev spaces above are defined as follows. We first set 
		$L^2_0(\overline E_{(x_0,y_0)})$ to be the space of measurable sections of 
		$\overline E_{(x_0,y_0)}\to T^+_{(x_0,y_0)}M$ which are $L^2$ with respect to the measure $d{\rm vol}_0:=t^{-1}dt\,d\xi$ and  define
		\[
		H^{k}_0(\overline E_{(x_0,y_0)})=\left\{u:(t\partial_t)^j(t\partial_\xi)^\beta u\in L^2_0(\overline E_{(x_0,y_0)})\right\},
		\]  
		and finally set
		\[
		t^\delta H^{k}_0(\overline E_{(x_0,y_0)})=\left\{u: t^{-\delta}u\in H^{k}_0(\overline E_{(x_0,y_0)})\right\}.
		\]
		An entirely similar procedure yields weighted $0$-Sobolev spaces $x^\delta H^{k}_0(\overline F)$, where $\overline F$ is a vector bundle over an ALH space $(M,g)$ endowed with a special defining function $x$. In this more general context we set  
		$x^\delta H^{k}_0:=x^\delta H^{k}_0(\underline{\mathbb R})$. 
		
		\begin{remark}\label{emb:theo}
			The weighted H\"older and Sobolev spaces introduced above satisfy suitable embedding theorems and we refer to \cite[Lemma 8]{usula2021yang} for a precise description. These embeddings will be extensively  used below without further notice.  
		\end{remark}
		
		Besides the normal operators above, we also 		
		make use of a simpler ``model'', the {\em indicial family}, which is given by 
		\[
		\mathcal I_{(x_0,y_0)}(\mathcal L)(\zeta)=\sum_{j\leq 2}a_{j,0}(0,0)\zeta^j, \quad \zeta\in\mathbb C. 
		\]		
		Note that this is a polynomial in $\zeta$ with coefficients in $C^\infty({\rm End}(\overline E_{(x_0,y_0)}))$. Actually, the ellipticity of $\mathcal L$  gives that $a_{(2,0)}(0,0)\in C^\infty({\rm Iso}(\overline E_{(x_0,y_0)}))$, so there are at most finitely many {\em indicial roots} (i.e. $\zeta\in\mathbb C$ such that $\mathcal I_{(x_0,y_0)}(\mathcal L)(\zeta)$ is {\em not} invertible). 
		
		In principle, the indicial roots may vary with $(x_0,y_0)$, which complicates the ensuing analysis considerably.
		Fortunately, this is never the case for geometric generalized Laplacians, as shown by the next result which follows readily from (\ref{norm:rep:def}); see also \cite[Remark 2.1]{fine2022non}. 
		
		\begin{proposition}\label{geom:lap:norm}
			If $\mathcal L$ is a {\em geometric} generalized Laplacian  and 
			$(x_0,y_0)\in X$, then $\mathcal N_{(x_0,y_0)}(\mathcal L)$ may be identified to the same operator acting on (the corresponding {geometric} bundle over) $\mathbb H^{n+1}$.
		\end{proposition}
		
		In particular, for this class of generalized Laplacians, the indicial roots do {\em not} depend on $(x_0,y_0)$. In general, however, in order to simplify the analysis we must impose the {\em constancy} of the indicial roots, which we take for granted from now on. Another simplifying assumption takes place if we assume that $\mathcal L$ is {\em formally self-adjoint}.  
		In this case, there holds $\mathcal I(\mathcal L)(\zeta)=\mathcal I(\mathcal L)(n-\overline{\zeta})$, so the indicial roots are symmetric about the line ${\rm Re}\,\zeta=n/2$.
		Assume further that there exists no indicial root $\zeta$ with ${\rm Re}\,\zeta=n/2$ and let $I_{\mathcal L}$ be the largest open interval containing $n/2$ for which there is no indicial root $\zeta$ with ${\rm Re}\,\zeta\in I_{\mathcal L}$; the radius of $I_{\mathcal L}$ is called the {\em indicial radius} in \cite{lee2006fredholm}. In a sense, the non-indicial interval $I_{\mathcal L}$ determines a putative ``Fredholm range'' for $\mathcal L$, which is confirmed by the next fundamental theorem, a rather special case of a general result proved in \cite{mazzeo1987meromorphic,mazzeo1988hodge}.

		\begin{theorem}\label{norm:cond:t}[Fredholmness for generalized Laplacians in the $0$-calculus]
			Let $(M,g)$ be ALH and let $\mathcal L$ be a (not necessarily geometric but formally self-adjoint) generalized Laplacian operator acting on sections of a vector bundle $E$ over $M$ and satisfying the assumptions above. 
			Then 
			\begin{equation}\label{L:holder}
				\mathcal L:x^\delta C^{k,\alpha}_{0}(\overline E)\to x^\delta C^{k-2,\alpha}_{0}(\overline E)
			\end{equation}
			is Fredholm (of index zero) for some $\delta\in I_{\mathcal L}$ if and only if 
			\begin{equation}\label{norm:cond}
				\mathcal N_{(x_0,y_0)}(\mathcal L):x^\delta H^{k}_0(\overline E_{(x_0,y_0)})\to x^\delta H^{k-2}_0(\overline E_{(x_0,y_0)})
			\end{equation}
			is invertible (for any $(x_0,y_0)$). In this case, the kernel of (\ref{L:holder}) does not depend on $\delta\in I_{\mathcal L}$ and actually coincides with $\ker \mathcal L|_{L^2(\overline E)}$. 
		\end{theorem}

		It is instructive to  check how this applies to the scalar Laplacian $\prescript{g}{}{\Delta}$, which is the prototypical example of a geometric and formally self-adjoint generalized Laplacian.  		
		In this case, Fredholmness  in appropriate weighted H\"older spaces may be established as follows.
		First, from (\ref{lap:ind}) we may take $I_{\prescript{g}{}{\Delta}}=(0,n)$. 
			By Proposition \ref{geom:lap:norm}, for any $(x_0,y_0)\in X$,
			\begin{equation}\label{norm:lap}
			\mathcal N_{(x_0,y_0)}(\prescript{g}{}{\Delta})=\prescript{g_{\mathbb H}}{}{\Delta}:x^\delta H^{k}_0(\overline{\mathbb H}^{n+1})\to x^\delta H^{k-2}_0(\overline{\mathbb H}^{n+1}),\quad \delta\in (0,n).
			\end{equation}
			Since for {\em any} ALH space $(M,g)$ there holds
			\[
			x^{\frac{n}{2}}H^0_{0}=L^2(M;d{\rm vol}_g),
			\]
			self-adjointness of $\prescript{g_{\mathbb H}}{}{\Delta}$ and a bit of $0$-regularity theory imply that the invertibility of (\ref{norm:lap})
			gets reduced to checking that 
		 $\ker \prescript{g_{\mathbb H}}{}{\Delta}\big|_{L^2}=\{0\}$, which follows from the well-known McKean-type  estimate 
			\begin{equation}\label{mckean:funct}
				\int_{\mathbb H^{n+1}}\left|\prescript{g_{\mathbb H}}{}{\nabla}u\right|^2 d{\rm vol}_{g_{\mathbb H}}\geq\frac{n^2}{4}\int_{\mathbb H^{n+1}}u^2d{\rm vol}_{g_{\mathbb H}}, \quad u\in C^{\infty}_{\rm cpt}(\mathbb H^{n+1}).
			\end{equation}  
	 	Hence, 
	 	\begin{equation}\label{lap:0:n}
	 	\prescript{g}{}{\Delta}:x^\delta C^{k,\alpha}_{0}(\underline{\mathbb R})\to x^\delta C^{k-2,\alpha}_{0}(\underline{\mathbb R}), \quad \delta\in(0,n),
	 	\end{equation}
	is Fredholm, which is a first step toward proving finer mapping properties like invertibility. In the case at hand, we may proceed by appealing  to the last assertion in Theorem \ref{norm:cond:t} to conclude that the obstruction to (\ref{lap:0:n}) being invertible lies in the non-triviality of $\ker \prescript{g}{}{\Delta}\big|_{L^2}$. 
		We now recall an useful result in this setting. 
		
		\begin{theorem}\label{mazz:top} \cite{mazzeo1988hodge}
			If $(M,g)$ is ALH and $k<n/2$ then $\ker \prescript{g}{}{\Delta}^k\big|_{L^2}=H^k(\overline M,Y)$, where $\prescript{g}{}{\Delta}^k$ is the Hodge Laplacian acting on $k$-forms. 
		\end{theorem}
		
		Thus, $\ker \prescript{g}{}{\Delta}\big|_{L^2}=H^0(\overline M,Y)=\{0\}$
		and we conclude that (\ref{lap:0:n}) is invertible.
		Essentially the same argument as above leads to the following useful result, which we record here for later reference. 		
		
		\begin{theorem}\label{th:usef:geo}	
			Let $\mathcal L$ be a geometric and formally self-adjoint generalized Laplacian on a ALH manifold and as always assume  that its indicial family yields a non-indicial interval $I_{\mathcal L}\neq\emptyset$; in the language of \cite{lee2006fredholm}, this means that the indicial radius of $\mathcal L$ is positive. Then 
			(\ref{L:holder}) is invertible for $\delta\in I_{\mathcal L}$ whenever the following conditions hold:
			\begin{enumerate}
				\item its normal operator, which does not depend on $(x_0,y_0)$ and may be identified to the corresponding geometric operator acting on $\mathbb H^{n+1}$, has a trivial kernel in $L^2$;
				\item ${\ker}\,\mathcal L\big|_{L^2}=\{0\}$. 
			\end{enumerate} 
			Moreover, the same conclusion holds for a (not necessarily geometric but formally self-adjoint) generalized Laplacian whenever its normal operators may be identified to the normal operator of a geometric generalized Laplacian for which $\rm (1)$ holds. 
		\end{theorem}
		
		\begin{remark}\label{non-deg:req}
			From the discussion above we see that 
			the steps required to check that a formally self-adjoint generalized Laplacian $\mathcal L$ is invertible may be summarized as follows. First, one  verifies that
			its normal operator $\mathcal N(\mathcal L)$ is well-behaved, which means that $I_{\mathcal L}\neq \emptyset$ and item (1) above holds. As explained in detail in \cite{mazzeo1987meromorphic,mazzeo1988hodge,albinnotes2008,hintz2021elliptic}, these conditions entail the complementary ellipticity around the conformal boundary we mentioned earlier and allow for the construction, via the  $0$-calculus, of a parametrix for $\mathcal L$ in this asymptotic region. Together with interior ellipticity, this turns out to be equivalent to $\mathcal L$ being Fredholm in the non-indicial range $\delta\in I_{\mathcal L}$. Granted this, invertibility is achieved if $\mathcal L$ is {\em non-degenerate} in the sense that item (2) in the previous theorem holds.  
		\end{remark}
		
		\section{Einstein-Yang-Mills fields}\label{eym:disc}
		
		The purpose of this section is to review the  theory of Einstein-Yang-Mills fields  with an emphasis on its variational features \cite{kerner1968generalization,cho1975higher,hermann1978yang,bourguignon1989mathematician,bleecker2005gauge}. This is a non-abelian generalization of the classical Kaluza-Klein unified field theory, the first serious attempt to put together the gravitational and electromagnetic fields in a single mathematical framework, and involves the consideration of a pair $(g,\omega)$, where $g$ is a metric on a smooth manifold and $\omega$ is a connection on a ``external'' principal bundle. Sometimes we refer to $(g,\omega)$ simply as a {\em configuration}.  
		
		We first consider the metric side of the theory.
		If $M^{n+1}$ is a smooth manifold, we denote by $\mathcal M_M$ the space of all (not necessarily ALH) smooth metrics on  $M$. 
		As usual, we denote by $\mathcal S^k(M)=C^\infty(S^kT^*M)$ the space of symmetric $k$-covariant tensors on $M$. Recall that we may view $\mathcal S^2(M)$ as the tangent space to $\mathcal M_M$ at some metric $g$,  $T_g\mathcal M_M= \mathcal S^2(M)$.
		
		In the presence of a metric $g\in\mathcal M_M$, there exists a linear map $\prescript{g}{}{\delta}:\mathcal S^2(M)\to \mathcal S^{1}(M)$, the divergence operator, given in index notation by
		\[
		\prescript{g}{}{\delta}\eta_{j}=-\prescript{g}{}{\nabla}^i\eta_{ij},
		\]
		where $\prescript{g}{}{\nabla}$ is the covariant derivative induced by the Levi-Civita connection of $g$ (and of course we raise indexes with respect to $g$). 
		Note that $\prescript{g}{}{\delta}g=0$. 
		The following well-known result is an immediate consequence of the definition.
		
		\begin{lemma}\label{div:adj}
			The divergence $\prescript{g}{}{\delta}$ is the formal adjoint of the map $\prescript{g}{}{\delta}^*:\mathcal A^1(M)\to \mathcal S^2(M)$ given by the symmetrization of $\prescript{g}{}{\nabla}$. In other words, $\prescript{g}{}{\delta}^*\eta=\frac{1}{2}\mathbb L_{\eta^\sharp}g$, where $\mathbb L_{\eta^\sharp}g$ is the Lie derivative of $g$ with respect to the vector field $\eta^\sharp$ dual to $\eta$. 
		\end{lemma}
		
		It is known that  
		\begin{equation}\label{bian:id}
			\prescript{g}{}{\delta}\prescript{g}{}{\rm G}=0,
		\end{equation}
		where 
		\[
		\prescript{g}{}{\rm G}=\prescript{g}{}{\rm Ric}-\frac{\prescript{g}{}{\rm R}}{2}g
		\]
		is the {\em Einstein tensor} of $g$. Here, ${\rm Ric}$ and ${\rm R}$ stand for the Ricci tensor and the scalar curvature, respectively. As a consequence, 
		the {\em Bianchi map} 
		\[
		\prescript{g}{}{\mathcal B}=\prescript{g}{}{\delta}+\frac{1}{2}d\prescript{g}{}{\rm tr}:\mathcal S^2(M)\to \mathcal A^{1}(M)
		\] 
		satisfies
		\begin{equation}\label{bian:id:2}
			\prescript{g}{}{\mathcal B}\left(\prescript{g}{}{\rm Ric}+ng\right)=0.
		\end{equation}

		We now turn to the Yang-Mills side of the theory.
		Let $P$ be a principal bundle over $M$ with structural group $L$, a compact Lie group whose Lie algebra $\mathfrak l$ is endowed with a fixed ${\rm Ad}$-invariant inner product $q$. One has a natural Lie algebra homomorphism $\mathfrak l\to \mathcal X(P)$, $X\mapsto X^*$, and composing this  
		with evaluation at $p\in P$ yields an identification
		\[
		\mathfrak l \equiv V_p=\ker \pi_*(p)\subset T_pP,
		\]
		the vertical subspace at $p$. 
		Here, $\pi:P\to M$ is the natural projection. 
		A {\em connection}  on $P$ is just an
		equivariant way of algebraically complementing the vertical
		distribution $V$, so that $T_pP=V_p\oplus H_p$ for some horizontal
		distribution $H$. Equivalently we can think of  a connection  as a
		$\mathfrak l$-valued $1$-form $\omega$ on $P$ satisfying $\ker
		\omega_p=H_p$, $\omega_p(X^*_p)=X$ and $R_l^* \omega=
		{\rm Ad}_{l^{-1}}\omega$, $l\in L$, where $R$ stands for the right $L$-action.
		
		Given $k\geq 0$ let
		$\mathcal A^k(P,{\mathfrak l})$ be the space of $\mathfrak l$-valued $k$-forms on $P$,
		that is,
		$\mathcal A^k(P,{\mathfrak l})=C^\infty(\Lambda^kT^*P\otimes {\underline {\mathfrak l}})$,
		where $\underline {\mathfrak l}=P\times \mathfrak l$ is the trivial Lie algebra bundle. 
		Notice that 
		connections belong to $\Ac^1(P,{\mathfrak l})$. We also consider the subspace
		$\mathcal A_{{\rm
				Ad}}^\bullet(P,\mathfrak l)\subset \mathcal A^k(P,{\mathfrak l})$ associated to the
		adjoint representation ${{\rm Ad}}:L\times\mathfrak l\to\mathfrak l$, whose  elements $\varphi$ are characterized by
		\begin{itemize}
			\item $\varphi$ is equivariant with respect to ${\rm Ad}$: $R^*_l\varphi={\rm Ad}_{l^{-1}}\varphi$, $l\in L$;
			\item $\varphi$ vanishes if some of its entries is vertical.
		\end{itemize}
		Thus, 
		$\mathcal A^k_{\rm Ad}(P,\mathfrak l)$ gets identified to 
		$C^\infty(\Lambda^kT^*M\otimes
		{\rm Ad}\,\mathfrak l)$, 
		the space of $k$-forms over $M$ with values in the associated adjoint bundle
		${\rm Ad}\,\mathfrak l=P\times_{\rm Ad}\mathfrak l$.
		
		The relevance of $\mathcal A^\bullet_{{\rm Ad}}(P,\mathfrak l)$ is primarily  due to the following facts. 
		First, the difference of any two connections lies in $\mathcal A^1_{{\rm Ad}}(P,\mathfrak l)$, so at least formally we may write $T_\omega \mathcal C_P=\mathcal A^1_{{\rm Ad}}(P,\mathfrak l)$, where $\omega\in \mathcal C_P$, the space of all connections on $P$.
		Also, if we  fix $\omega\in \mathcal C_P$
		then for 
		$\varphi\in\mathcal A^k(P,{\mathfrak l})$ we may define
		$\varphi^H\in\mathcal A^k(P,{\mathfrak l})$
		by
		$\varphi^H(X_1,\ldots,X_k)=\varphi(X_1^H,\ldots,X_k^H)$, where
		$X_i^H$ is the horizontal projection of $X_i$ with respect to  $\omega$, and
		then set $d_{\omega}\varphi:=(d\varphi)^H$. Upon restriction we obtain 
		\begin{equation}\label{d:om}
			d_{\omega}:\mathcal A^k_{{\rm Ad}}(P,\mathfrak l)\to \mathcal A^{k+1}_{{\rm Ad}}(P,\mathfrak l),
		\end{equation}
		the {\em covariant derivative} associated to $\omega$. 
		Also, if we fix a metric $g$ on $M$ then with the help of $q$  we may define a  $L^2$ inner product on $\mathcal A^\bullet_{{\rm Ad}}(P,\mathfrak l)$.
		
		\begin{lemma}\label{adj:d} 
			The corresponding formal adjoint of $d_\omega$ in (\ref{d:om}) is given by
			\[
			\prescript{g}{}{d}^*_\omega=-(-1)^{(n+1)k}\star d_{\omega}\star,
			\]  
			where $\star$ is the Hodge star operator. 
		\end{lemma}
		
		This allows us to define a {\em twisted} Hodge Laplacian
		\begin{equation}\label{hodge:conn}
			\prescript{g}{}{\Delta}_\omega^\bullet:=d_\omega \prescript{g}{}{d}^*_\omega+\prescript{g}{}{d}^*_\omega d_\omega:\mathcal A^\bullet_{{\rm Ad}}(P,\mathfrak l)\to \mathcal A^\bullet_{{\rm Ad}}(P,\mathfrak l) 
		\end{equation}
		which preserves the grading. Notice that $\prescript{g}{}{\Delta}_\omega^\bullet$ is a generalized Laplacian, albeit not a geometric one in the sense of Definition \ref{geom:op}, due to its dependence on $\omega$. 
		
		The last ingredient needed to set up the Einstein-Yang-Mills variational theory is the {\em curvature} associated to a connection:		
		$$
		\Omega_{\omega}=d_{\omega}\omega\in \mathcal A^2_{{\rm Ad}}(P,\mathfrak l).
		$$
		This allows us to define a {\em self-action density} $Q=Q(g,\omega)$
		as follows. First, we define a $4$-covariant tensor on $P$ by setting 
		\[
		q(\Omega_\omega)(X,Y,Z,W)=q(\Omega_\omega(X,Y),\Omega_\omega(Z,W)), \quad X,Y,Z,W\in \mathcal X(M),
		\]
		and then extending this to $P$ in the expected manner. If in  terms of a horizontal frame $\{e_i\}$ we set $q(\Omega_\omega)_{ijkm}=q(\Omega_\omega)(e_i,e_j,e_k,e_m)$ then  
		\[
		Q(g,\omega):=-\frac{1}{4}g^{ik}g^{jm} q(\Omega_\omega)_{ijkm}. 
		\] 
		In the language of \cite[Section 10.2]{bleecker2005gauge}, $Q$ is the gauge invariant density induced via Utiyama's theorem by the fixed ${\rm Ad}$-invariant inner product $q$ on $\mathfrak l$, viewed as a (quadratic) curvature Lagrangian. 
		
		We may now define the {\em Einstein-Yang-Mills functional}, $\mathcal Y:\mathcal M_M\times \mathcal C_P\to\mathbb R$, by
		\begin{equation}\label{eym:action}
			\mathcal Y(g,\omega)=\int_M\left(\prescript{g}{}{\rm R}+n(n-1)+Q(g,\omega)\right)d{\rm vol}_g.
		\end{equation}
		
		\begin{definition}\label{eym:fields}
			A configuration $(g,\omega)\in \mathcal M_M\times \mathcal C_P$ is an {\em Einstein-Yang-Mills (EYM) field} if it is critical for $\mathcal Y$.
		\end{definition}
		
		\begin{proposition}\cite[Theorem 9.3.3]{bleecker2005gauge}\label{blee:eym}
			$(g,\omega)\in \mathcal M_M\times \mathcal C_P$ is  EYM if and only if it satisfies
			\begin{equation}\label{blee:eym:eq}
				\left\{
				\begin{array}{rcl}
					\prescript{g}{}{\rm G} - \frac{n(n-1)}{2}g & = & K_{g,\omega}\\
					&&\\
					\prescript{g}{}{d}^*_\omega\Omega_\omega & = & 0
				\end{array}
				\right.
			\end{equation}
			where
			\begin{equation}\label{exp:K}
				{K_{g,\omega}}_{ij}=\frac{1}{2}g^{km}q(\Omega_\omega)_{ikjm}+
				\frac{1}{2}Q(g,\omega)g_{ij}
			\end{equation}
			is the {\em stress-energy tensor}.
		\end{proposition}
		
		\begin{proof}
			This follows from the identities 
			\begin{equation}\label{gaug:met}
				\frac{d}{dt}\mathcal Y(g+th,\omega)\big|_{t=0}=\int_M\left\langle-\left(\prescript{g}{}{\rm G} - \frac{n(n-1)}{2}g\right) +K_{g,\omega},h\right\rangle d{\rm vol}_g,
			\end{equation}
			holding for any compactly supported $h\in \mathcal S^2(M)$, and 
			\begin{equation}\label{gaug:con}
				\frac{d}{dt}\mathcal Y(g,\omega+ta)\big|_{t=0}=-\int_M\left\langle \prescript{g}{}{d}_\omega^*\Omega_\omega,a\right\rangle d{\rm vol}_g, 
			\end{equation}
			which holds for any  $a\in\mathcal A^1_{\rm Ad}(P,\mathfrak l)$ whose ``projection'' on $M$ is compactly supported.
		\end{proof}
		
		\begin{remark}\label{kal:klein}
			If $P=M\times\mathbb S^1$, we recover the Kaluza-Klein field equations.
		\end{remark}

		For our purposes, and having (\ref{bian:id:2}) in mind, it is convenient to rewrite the EYM field equations as 
		\begin{equation}\label{eym:field:rew}
			\left\{
			\begin{array}{rcl}
				\prescript{g}{}{\rm Ric}+ng & = & \widetilde{K}_{g,\omega}\\
				&&\\
				\prescript{g}{}{d}^*_\omega\Omega_\omega & = & 0
			\end{array}
			\right.
		\end{equation}
		where
		\begin{equation}\label{exp:K:tilde}
			\widetilde{K}_{g,\omega}={K}_{g,\omega}-\frac{\kappa_{g,\omega}}{n-1}g, \quad \kappa_{g,\omega}=\prescript{g}{}{\rm tr}\,{K}_{g,\omega}
		\end{equation}
		is the {\em modified stress-energy tensor}. 
		
		\begin{remark}\label{conf:inv}
			From (\ref{exp:K}) we get
			\[
			\kappa_{g,\omega}=\frac{n-3}{2}Q(g,\omega),
			\]
			so that 
			\begin{equation}\label{exp:K:tilde:2}
				\widetilde{K}_{g,\omega}=K_{g,\omega}-\frac{n-3}{2(n-1)}Q(g,\omega)g=\frac{1}{2}g^{km}q(\Omega_\omega)_{\cdot k\cdot m}-\frac{1}{n-1}Q(g,\omega)g.
			\end{equation}
			Also, it follows from (\ref{blee:eym:eq}) that a EYM field $(g,\omega)$ satisfies 
			\[
			\prescript{g}{}{\rm R}=-n(n+1)-\frac{2\kappa_{g,\omega}}{n-1}. 
			\] 
			In particular, $\kappa_{g,\omega}=0$ and $\prescript{g}{}{\rm R}=-12$ if $n=3$, which reflect the conformal invariance of the Yang-Mills self-action in dimension $n+1=4$ \cite{parker1982gauge}. 
		\end{remark}
		
		\begin{remark}\label{K:quad}
			There are two points regarding $\widetilde K_{g,\omega}$ which follow immediately from (\ref{exp:K:tilde:2}) and are worth mentioning here:
			\begin{itemize}
				\item its dependence on $\omega$ takes place through the curvature $\Omega_\omega$;
				\item this dependence is {\em quadratic} in $\Omega_\omega$. 
			\end{itemize}
		\end{remark}
		
		The next result  explores the ``local'' conservation law in (\ref{bian:id}) and provides a differential identity satisfied by $\widetilde{K}_{g,\omega}$ which plays an important role in the proof of our main result.  
		
		\begin{lemma}\label{usef:lem}
			There holds $\prescript{g}{}{\mathcal B}\widetilde{K}_{g,\omega}=0$ for {\em any} $(g,\omega)\in \mathcal M_M\times \mathcal C_P$.
		\end{lemma}
		
		\begin{proof}
			This is just the well-known Noether's variational principle applied to our setting \cite{kosmann2012noether};  compare with \cite[Theorem 9.2.17]{bleecker2005gauge}.
			Indeed, for any  $X\in\mathcal X(M)$ it follows from (\ref{exp:met}) below that  
			$h=\mathbb L_Xg$ comes from a deformation induced by a local flow of diffeomorphims, so the left-hand side of (\ref{gaug:met}) vanishes due to diffeomorphism invariance.  Thus, we may use Lemma \ref{div:adj} and (\ref{bian:id}) to check that  ${K}_{g,\omega}$ is ``locally'' conserved:
			\[
			\prescript{g}{}{\delta}{K}_{g,\omega}=	\prescript{g}{}{\delta}\left(\prescript{g}{}{\rm G} - \frac{n(n-1)}{2}g\right)=0.
			\]
			We now compute using (\ref{exp:K:tilde}):
			\begin{eqnarray*}
				\prescript{g}{}{\mathcal B}\widetilde {K}_{g,\omega}
				&=& \left(\prescript{g}{}{\delta}+\frac{1}{2}d\prescript{g}{}{\rm tr}\right)\left({K}_{g,\omega}-\frac{\kappa_{g,\omega}}{n-1}g\right)\\
				& = & \frac{d\kappa_{g,\omega}}{n-1}+\frac{d\kappa_{g,\omega}}{2}-\frac{n+1}{2(n-1)}{d\kappa_{g,\omega}}\\
				& = & 0
			\end{eqnarray*}
		\end{proof}
		
		\begin{remark}\label{off-shell}
			It is immediate from (\ref{bian:id:2}) that the conclusion of Lemma \ref{usef:lem} holds whenever $(g,\omega)$ is a EYM field. But the crucial point here, which turns out to be essential for the applications we have in mind, is that it actually holds {\em off-shell}, that is, irrespective of the configuration being a solution of (\ref{eym:field:rew}). 
		\end{remark}

		We now observe that the action (\ref{eym:action}) admits a large gauge symmetry group which, as is well-known, spoils the ellipticity of the system of equations defined by (\ref{eym:field:rew}). 
		First, the gauge group 
		$\mathpzc L_P$ of $P$, formed by
		automorphisms of $P$ (i.e diffeomorphism $f:P\to P$ commuting with the right $L$-action) and
		inducing the identity on $M$, acts naturally on $\mathcal C_P$ by pull-back.
		Actually, $\mathpzc L_P=C^\infty({\rm Ad}\,L)$, where ${\rm Ad}\,L=P\times_{\rm Ad}L$ is the Lie group bundle associated to the adjoint action of $L$ on itself.
		We also observe that the group ${\rm Diff}(M)$ of diffeomorphisms of $M$ acts naturally on $\mathcal M_M$ by pullback. 	Combining these we obtain an action of the {\em gauge group} 
		\begin{equation}\label{gauge:produ}
			\mathscr G:={\rm Diff}^o(M)\times \mathpzc L_P
		\end{equation}
		on the configuration space $\mathcal M_M\times \mathcal C_P$, where  ${\rm Diff}^o(M)\subset {\rm Diff}(M)$ is the connected component of the identity.
		Concretely, $(g,\omega)\cdot\Phi=(\Phi_1^*g,\Phi_2^*\omega)$, $(g,\omega)\in \mathcal M_M\times\mathcal C_P$, $\Phi=(\Phi_1,\Phi_2)\in\mathscr G$. Clearly, this preserves the EYM action (\ref{eym:action}).

		We may now explain how this gauge action of $\mathscr G$ on $\mathcal M_M\times \mathcal C_P$ spoils ellipticity of (\ref{eym:field:rew}) at the infinitesimal level. First, the Lie algebra of $\mathpzc L_P$ is  $\mathpzc l_P=C^\infty({\rm Ad}\,\mathfrak l)$, where ${\rm Ad}\,\mathfrak l=P\times_{\rm Ad}\mathfrak l$ is the Lie algebra bundle associated to the adjoint representation, 
		and there
		exists an exponential map ${\rm Exp}:{\mathpzc l}_P\to \mathpzc L_P$ induced by $\exp:\mathfrak l\to L$.
		For $\mathfrak a\in \mathpzc l_P$ we thus have
		\begin{equation}\label{exp:con}
			\frac{d}{dt} \omega\cdot{{\rm Exp}}\,t\mathfrak a\big|_{t=0}=d_\omega \mathfrak a, \quad \omega\in \mathcal C_P.
		\end{equation}
		Similarly, the Lie algebra of ${\rm Diff}^o(M)$ is $\mathcal X(M)$, the space of vector fields on $M$, with the exponential map ${\mathsf{ Exp}}:\mathcal X(M)\to {\rm Diff}^o(M)$ being such that $t\mapsto {\mathsf{ Exp}}(tX)$ corresponds to flowing along the integral curves of $X$. Hence,
		\begin{equation}\label{exp:met}
			\frac{d}{dt}g\cdot{{\mathsf{Exp}}}\,tX\big|_{t=0}=\mathbb L_Xg, \quad g\in\mathcal M_M.
		\end{equation}
		Now, the linearization of the map $\omega\mapsto \prescript{g}{}{d}^*_\omega\Omega_\omega$, helding $g$ fixed and in the direction of $a\in \mathcal A^1_{\rm Ad}(P,\mathfrak l)$,  is the map 
		\begin{equation}\label{exp:con:2}
			a\mapsto 
			\prescript{g}{}{d}^*_\omega d_\omega a+a^{\sharp}\iprod\Omega_\omega=\prescript{g}{}{\Delta}_\omega^1a-d_\omega\prescript{g}{}{d}^*_\omega a+a^{\sharp}\iprod\Omega_\omega. 
		\end{equation}
		Although the twisted Hodge Laplacian $\prescript{g}{}{\Delta}_\omega^1$ is a 
		generalized Laplacian, hence elliptic, the remaining second order term $a\mapsto -d_\omega\prescript{g}{}{d}^*_\omega a$ spoils ellipticity 
		and reflects the infinitesimal effect of the gauge action of $\mathpzc L_P$ on $\mathcal C_P$ as described in (\ref{exp:con}) (with $\mathfrak a=d^*_\omega a$).
		Also, it is known that the linearization of the ``Poincar\'e-Einstein map'' $g\mapsto \prescript{g}{}{\rm Ric}+ng$ in the direction of $A\in\mathcal S^2(M)$ is the map 
		\begin{equation}\label{exp:met:2}
			A\mapsto \prescript{g}{}{\mathscr L}A:= \frac{1}{2}\left(\prescript{g}{}{\Delta}_LA+ 2n A-\mathbb L_{(\prescript{g}{}{\mathcal B}A)^\sharp}g\right).
		\end{equation}
		Although the Lichnerowicz Laplacian $\prescript{g}{}{\Delta}_L$, defined in Section \ref{pert:lich} below, is a geometric generalized Laplacian, hence elliptic, the remaining second order term $A\mapsto -\mathbb L_{(\prescript{g}{}{\mathcal B}A)^\sharp}g$ again spoils ellipticity and reflects the infinitesimal effect of the gauge action of ${\rm{Diff}}^o(M)$ on $\mathcal M_M$ as described in (\ref{exp:met}) (with $X=(\prescript{g}{}{\mathcal B}A)^\sharp$).
		
		\begin{remark}\label{inc:discuss}
			Upon comparison of the Poincar\'e-Einstein map above with the first field equation in (\ref{eym:field:rew}), one notes that the modified stress-energy tensor $\widetilde K_{g,\omega}$ has been overlooked. The reason for this is that our linearization analysis will always take place around a {\em trivial} configuration $(g,\omega)$, which means in particular that $\omega$ is flat ($\Omega_\omega=0$); see Definition \ref{nondeg:conf}.  Since, as already observed in Remark \ref{K:quad}, $\widetilde K_{g,\omega}$ depends quadratically on $\Omega_\omega$, we see that this quantity does not contribute to the linearizations (both with respect to $g$ and $\omega$) at such $(g,\omega)$; compare with the proof of Lemma \ref{differ:cal}.  
		\end{remark}
		
		The discussion above suggests that 
		exponentiation of $\ker \prescript{g}{}{\mathcal B}\oplus\ker \prescript{g}{}{d}_\omega^*$ should possibly yield a local slice for the natural action of the gauge group $\mathscr G$ on $\mathcal M_M\times \mathcal C_P$ around $(g,\omega)$. 
		Indeed, this works fine whenever $\prescript{g}{}{\rm Ric}<0$ and leads to 
		the {\em Bianchi-Coulomb gauge}, an amalgamation of the gauge fixings in \cite{biquard2000metriques,mazzeo2006maskit,usula2021yang} which
		effectively eliminates the troublesome terms in  (\ref{exp:con:2}) and (\ref{exp:met:2}), thus restoring ellipticity; see Section \ref{fix:gauge}.
		
		\begin{remark}\label{gaug:con:0}
			As yet another manifestation of Noether's variational principle, 
			if we take 
			$a=d_\omega\mathfrak a$, $\mathfrak a\in \mathpzc l_P=\mathcal A^0_{\rm Ad}(P,\mathfrak l)$, in (\ref{gaug:con}) it follows from (\ref{exp:con}) that the left-hand side vanishes due to gauge invariance and we thus obtain the useful identity
			\begin{equation}\label{off:shell:con}
				\prescript{g}{}{d}_\omega^*\prescript{g}{}{d}_\omega^*\Omega_\omega=0,
			\end{equation}
			which holds  {\em off-shell}; compare with Remark \ref{off-shell}. 
		\end{remark}
		
		\begin{remark}\label{ident:gauge}
			Recall that a principal bundle $P\to M$ with structural Lie group $L$ gives rise to the exact sequence
			\[
			{\rm Id}\longrightarrow  \mathpzc L_P \stackrel{\bf i}{\longrightarrow} {\rm Aut}(P)\stackrel{\bf p}{\longrightarrow} {\rm Diff}(M)
			\]
			where ${\rm Aut}(P)$ is the group of automorphisms of $P$ (hence, each element of ${\rm Aut}(P)$ induces a diffeomorphism $f:M\to M$). 
			Passing to the connected component of the identity we get the exact sequence
			\begin{equation}\label{gauge:exact}
				{\rm Id}\longrightarrow  \mathpzc L_P \stackrel{\bf i}{\longrightarrow} {\rm Aut}^o(P)\stackrel{\bf p}{\longrightarrow} {\rm Diff}^o(M)\longrightarrow {\rm Id}
			\end{equation}
			where ${\rm Aut}^o(P)={\bf p}^{-1}({\rm Diff}^o(M))$.
			If $P$ is trivial and  $P=M\times L$ is a fixed trivialization then the monomorphism
			\begin{equation}\label{split:mono}
				\alpha:{\rm Diff}^o(M)\to {\rm Aut}^o(P), \quad \alpha(f)(m,l)=(f(m),l), \quad (m,l)\in M\times L, 
			\end{equation}
			produces an splitting of (\ref{gauge:exact}), which means that ${\bf p}\circ\alpha={\rm Id}$. Thus, we obtain a semi-direct product representation 
			\begin{equation}\label{semi:dir}
				{\rm Aut}^o(P)=\mathpzc L_P \rtimes_\beta {\rm Diff}^o(M),
			\end{equation}	
			for a homomorphism $\beta:{\rm Diff}^o(M)\to {\rm Aut}(\mathpzc L_P)$.
			We now observe that triviality also  implies
			${\rm Ad}\, L=M\times L$, which gives an identification of $\mathpzc L_P$ with $({\rm Map}(M,L),\ast)$,
			the group of smooth maps from $M$ to $L$. In this representation,  $\beta(f)(\sigma)=\sigma\circ f^{-1}$, $\sigma\in {\rm Map}(M,L)$; compare with \cite[Proposition 1]{trautman2005groups}. Thus, the product rule in (\ref{semi:dir}) is 
			\[
			\left(\left(\sigma,f\right),\left(\sigma',f'\right)\right)\mapsto 
			\left(\sigma \ast (\sigma'\circ f^{-1}),f\circ f'\right).
			\]  
			Upon comparison with (\ref{gauge:produ}) we see that, at least in case $P$ is trivial, which is the only case where our main result (Theorem \ref{main}) effectively applies (see Remark \ref{non:flat} and Theorem \ref{witten}), and in the presence of a given trivialization, our choice of gauge group corresponds to an ``untwisted'' version of (\ref{semi:dir}), in which the natural action of ${\rm Diff}^o(M)$ on $\mathpzc L_P$ is not taken into account. This choice may be justified as follows. It is known that each configuration $(g,\omega)$ naturally determines a so-called {\em bundle metric}, say $\mathsf G_{g,\omega}$, on $P$ with respect to which $L$ acts by isometries \cite[Definition 9.3.4]{bleecker2005gauge}. Moreover, the scalar curvature of $\mathsf G_{g,\omega}$ equals the full Lagrangian density in (\ref{eym:action}) up to an additive constant \cite[Theorem 9.3.7]{bleecker2005gauge}. In this way, the EYM field equations (\ref{blee:eym:eq}), which couples the gravitational and Yang-Mills degrees of freedom, arise as a consequence of a single geometric variational principle associated to the   ``bundle action'' 
			\[\mathsf G_{g,\omega}\in \mathcal M_P^{b}\mapsto \int_P\left(R_{\mathsf G_{g,\omega}}+c\right)d{\rm vol}_{\mathsf G_{g,\omega}},
			\]
			where $\mathcal M_P^{b}$ is the space of bundle metrics of $P$; this is the modern interpretation of the original Kaluza-Klein approach to unification as explained in \cite{hermann1978yang,bourguignon1989mathematician,bleecker2005gauge}. Clearly, our gauge action of $\mathscr G={\rm Diff}^o(M)\times \mathpzc L_P$ on configurations induces gauge transformations on $\mathcal M_P^{b}$ 
			so that the bundle action above remains constant along the corresponding orbits.
			Thus, fixing a gauge for 
			the original gauge action of $\mathscr G$ on $\mathcal M_M\times\mathcal C_P$
			amounts to fixing a gauge for this latter bundle action and corresponds to taking ``genuine'' deformations of bundles metrics (i.e. which slice across the orbits). 
		\end{remark}

		\section{Some invertible operators}\label{inv:oper}
		
		Here we describe how Theorems \ref{norm:cond:t} and \ref{th:usef:geo} can be used to establish the invertibility of certain generalized Laplacians and then draw a few consequences of relevance in what follows. We emphasize that most of the computations and results below already appear in the available literature (see \cite{graham1991einstein,lee2006fredholm,mazzeo1988hodge,mazzeo2006maskit,usula2021yang}, for instance) but for the reader's convenience we have chosen to reproduce them here in some detail. In what follows, $(M,g)$ is taken to be a ALH manifold endowed with a special defining function $x$ near infinity.
		
		\subsection{Perturbations of the Lichnerowicz Laplacian}\label{pert:lich} We recall that the Lichnerowicz Laplacian, acting on $\mathcal S^2(M)$,  
		is given by
		\[
		\prescript{g}{}{\Delta}_L=\nabla^*\nabla+ 2(\prescript{g}{}{\widetilde {\rm Ric}}-\prescript{g}{}{\widetilde {\rm Riem}}),
		\]
		where $\nabla^*\nabla$ is the Bochner Laplacian,
		\[
		(\prescript{g}{}{\widetilde {\rm Ric}}\,h)_{jl}=\frac{1}{2}\left(\prescript{g}{}{\rm Ric}_{jk}h^k_l+\prescript{g}{}{\rm Ric}_{lk}h^k_j\right),
		\]
		\[
		(\prescript{g}{}{\widetilde {\rm Riem}}\,)_{jl}=
		\prescript{g}{}{\rm Riem}_{ijkl}h^{ik},
		\]
		and ${\rm Riem}$ stands for the Riemann tensor. If $(M^{n+1},g)$ is ALH we consider 
		\[
		\prescript{g}{}{{\Delta}}_{(n)}=\prescript{g}{}{\Delta}_L+2n,
		\] 
		a geometric and formally self-adjoint generalized Laplacian  that already appears in (\ref{exp:met:2}) and also will play a central role below. 
		Our aim here is to indicate the proof of the following fundamental result.
		
		\begin{theorem}\label{grah:lee:fred} \cite{graham1991einstein}
			If $(M,g)$ is ALH then the operator
			\begin{equation}\label{grah:lee:fred:2}
				\prescript{g}{}{\Delta_{(n)}}:x^\delta C_0^{k,\alpha}(S^2\prescript{0}{}{T^*M})\to x^\delta C_0^{k-2,\alpha}(S^2\prescript{0}{}{T^*M}), \quad 0<\delta<n,
			\end{equation}
			is Fredholm. 
		\end{theorem}
		
		Although we will not need this stronger version here, we prove below a more general result in order to illustrate the effectiveness of the $0$-calculus. 
		
		\begin{theorem}\label{grah:lee:fred:g} \cite[Proposition E]{lee2006fredholm}
			If $(M,g)$ is ALH and 
			\[
			\mu>n-\frac{n^2}{8}
			\]
			then the  operator
			\[
			\prescript{g}{}{\Delta}_{(\mu)}:=\prescript{g}{}{\Delta}_L+2\mu:x^\delta C_0^{k,\alpha}(S^2\prescript{0}{}{T^*M})\to x^\delta C_0^{k-2,\alpha}(S^2\prescript{0}{}{T^*M}), \quad \delta^{[3]}_-<\delta<\delta^{[3]}_+,
			\]
			is Fredholm, where  
			\begin{equation}\label{delta:3:pm}
				\delta^{[3]}_\pm=\frac{1}{2}\left(n\pm\sqrt{n^2+8(\mu-n)}\right).
			\end{equation}
		\end{theorem}
		
		
		The first step is to find a non-indicial interval for $\prescript{g}{}{\Delta}_{(\mu)}$.  
		Below we shall extensively use both the computation in (\ref{lap:ind}), which shows that the indicial family of the scalar Laplacian is $I({\prescript{g}{}{\Delta}})(\zeta)=-\zeta^{2}+n\zeta$,  and that near infinity we have
		\begin{equation}\label{ric:inf}
			(\prescript{g}{}{\widetilde{\rm Ric}}\,h)_{ij}={-nh_{ij}}+o(1)
		\end{equation}
		and 
		\begin{equation}\label{riem:inf}
			(\prescript{g}{}{\widetilde{\rm Riem}}\,h)_{ij}=-({}{g_{kl}}g_{ij}-g_{kj}g_{il})h^{kl}+o(1).
		\end{equation}
		
		We observe the decomposition 
		\[
		\mathcal S^2(M)=\mathcal S^2_{(tl)}(M)\oplus C^\infty(M)g, \quad h=h^{(tl)}+\frac{\prescript{g}{}{\rm tr}\,h}{n+1}g,
		\]
		where $h^{(tl)}$ is the trace-less piece of $h$. 
		Clearly, 
		$\nabla^*\nabla(\phi g)=(\prescript{g}{}{\Delta}\phi)g$ and a further computation using (\ref{ric:inf}) and (\ref{riem:inf}) shows that $\prescript{g}{}{\widetilde{\rm Ric}}-\prescript{g}{}{\widetilde{\rm Riem}}=o(1)$ on $C^\infty(M)g$. Hence, 
		\[
		\prescript{g}{}{\Delta}_{(\mu)}\big|_{C^\infty(M)g}(\phi g)=\left((\prescript{g}{}{\Delta}+2\mu)\phi\right)g+o(1),
		\]
		so by (\ref{lap:ind}) the indicial roots of $\prescript{g}{}{\Delta}_{(\mu)}|_{C^\infty(M)g}$ are
		\[
		\delta^\infty_{\pm}=\frac{1}{2}\left(n\pm\sqrt{n^2+8\mu}\right).
		\]
		
		We now look at the restriction of $\prescript{g}{}{\Delta}_{(\mu)}$ to $\mathcal S^2_{(tl)}(M))$. Following \cite[Section 2.2]{mazzeo2006maskit} we fix coordinates $z=(x,y_1,\dots,y_n)$ near infinity and decompose $\mathcal S^2_{(tl)}(M)$ accordingly.
		\begin{itemize}
			\item 
			The first factor, say $\mathcal S^{2,[1]}_{(tl)}(M)$, is generated by 
			\[
			u^{[00]}=n\frac{dx\otimes dx}{x^2}-\sum_i\frac{dy_i\otimes dy_i}{x^2}, \quad i=1,\dots,n.
			\]
			A computation shows that 
			\[
			\prescript{g}{}{\Delta}_{(\mu)}(\phi_{00} u^{[00]} )=\left(\left(\prescript{g}{}{\Delta}+2\mu\right)\phi_{00}\right)u^{[00]}+o(1). 
			\]
			Thus, the indicial roots of $\prescript{g}{}{\Delta}_{(\mu)}|_{\mathcal S^{2,[1]}_{(tl)}(M)}$ are
			\[
			\delta_{\pm}^{[1]}=\frac{1}{2}\left(n\pm\sqrt{n^2+8\mu}\right).
			\]
			\item 	The next factor, say $\mathcal S^{2,[2]}_{(tl)}(M)$, is generated by 
			\[
			u^{[0i]}:=\frac{dx\otimes dy_i}{x^2}.
			\]
			If $	u=\sum_i\phi_{i}u^{[0i]}$
			a computation shows that 
			\[
			\prescript{g}{}{\Delta}_{(\mu)} u=\sum_i\left(\left(\prescript{g}{}{\Delta}+2\mu-n+1\right)\phi_i\right)u^{[0i]}+o(1), 
			\]
			so
			the indicial roots of $\prescript{g}{}{\Delta}_{(\mu)}|_{\mathcal S^{2,[2]}_{(tl)}(M)}$ are
			\[
			\delta_{\pm}^{[2]}=\frac{1}{2}\left(n\pm\sqrt{n^2+4(2\mu-n+1)}\right).
			\]
			\item 	
			The last factor, say $\mathcal S^{2,[3]}_{(tl)}(M)$, is generated by 
			\[
			u^{[11]}:=(n-1)\frac{dy_1\otimes dy_1}{x^2}-\sum_{\alpha}\frac{dy_\alpha\otimes dy_\beta}{x^2},  
			\]
			\[
			u^{[\alpha\beta]}:=\frac{dy_\alpha\otimes dy_\beta}{x^2},\quad \alpha\neq\beta,
			\]
			and 
			\[
			\quad	u^{[1\alpha]}:=\frac{dy_1\otimes dy_\alpha}{x^2},
			\]
			where $\alpha, \beta=2,\dots,n$. 	If $	u=\sum_{ij}\phi_{ij}u^{[ij]}$
			a computation shows that 
			\[
			\prescript{g}{}{\Delta}_{(\mu)} u=\sum_{ij}\left((\prescript{g}{}{\Delta}+2(\mu-n))\phi_{ij}\right)u^{[ij]}+o(1),
			\]
			so the indicial roots of $\prescript{g}{}{\Delta}_{(\mu)}\big|_{\mathcal S^{2,[3]}_{(tl)}(M)}$ are $\delta^{[3]}_\pm$ as defined in (\ref{delta:3:pm}).
		\end{itemize}
		
		The aftereffect of the computation above is that we may take $I_{\prescript{g}{}{{\Delta}}_{(\mu)}}=(\delta^{[3]}_-,\delta^{[3]}_+)$, which is realized for $\prescript{g}{}{\Delta}_{(\mu)}\big|_{\mathcal S^{2,[3]}_{(tl)}(M)}$. 
		
		In order to complete the proof of  Theorem \ref{grah:lee:fred:g} via Theorem \ref{norm:cond:t}, we must check that the normal operator of $\prescript{g}{}{{\Delta}}_{(\mu)}$ is invertible. Since $\prescript{g}{}{{\Delta}}_{(\mu)}$ is (obviously) geometric, it follows from Proposition \ref{geom:lap:norm} that its normal operator agrees with the corresponding geometric operator acting on $\mathbb H^{n+1}$, so that by the argument leading to Theorem \ref{th:usef:geo} it suffices to check that the $L^2$ kernel of this latter operator is trivial. Since 
		\[
		\prescript{g}{}{\Delta}_{(\mu)}\big|_{\mathcal S^{2}_{(tl)}(\mathbb H^{n+1})}=\nabla^*\nabla+2(\mu-n-1), \quad 
		\prescript{g}{}{\Delta}_{(\mu)}\big|_{C^\infty(\mathbb H^{n+1})g_{\mathbb H}}=\prescript{g_{\mathbb H}}{}{\Delta}+2\mu,
		\] 
		this follows from (\ref{mckean:funct}) and the  
		{McKean-type estimate}
		\[
		\int_{\mathbb H^{n+1}}\langle\nabla^*\nabla h,h\rangle d{\rm vol}_{g_{\mathbb H}}\geq\left(\frac{n^2}{4}+{2}\right)\int_{\mathbb H^{n+1}}|h|^2 d{\rm vol}_{g_{\mathbb H}}, \quad h\in C^\infty_{\rm cpt}(\mathcal S_{(tl)}^{2}(\mathbb H^{n+1}));
		\]
		compare with \cite[Lemma 7.12]{lee2006fredholm}. This completes the proof of  Theorem \ref{grah:lee:fred:g}.
		
		For later reference we record an immediate consequence of the argument above. Let us recall a central concept in the deformation theory of PE metrics.
		
		\begin{definition}\label{nondeg:pe}
			A PE manifold $(M,g)$ is {\rm non-degenerate} if $\ker \prescript{g}{}{{\Delta}}_{(n)}\big|_{L^2}$ is trivial. 
		\end{definition}
		
		\begin{corollary}[of the proof]\label{cor:rel}
			If a PE manifold $(M,g)$ is non-degenerate then (\ref{grah:lee:fred:2}) is invertible. In particular, this happens for $(M,g)=(\mathbb H^{n+1},g_{\mathbb H})$. 
		\end{corollary}

		\subsection{The Hodge Laplacian twisted by a connection}\label{hod:twist}
		
		We first look at the mapping properties of the {\em untwisted} Hodge Laplacian
		\[
		\prescript{g}{}{\Delta}^1:\mathcal A^1(M)\to\mathcal A^1(M),
		\] 
		which is obtained from (\ref{hodge:conn}) by taking $P=M\times\mathbb S^1$ endowed with the canonical flat connection. Notice that $\prescript{g}{}{\Delta}^1$ is geometric and formally self-adjoint. As usual, we first determine a non-indicial interval for $\prescript{g}{}{\Delta}^1$. Near infinity we decompose 
		\[
		\Lambda^1 T^*M=\Lambda^1T^*X\oplus \langle dx\rangle
		\]  
		and, with respect to this decomposition,
		\begin{equation}\label{ind:delta1}
			\prescript{g}{}{\Delta}^1=
			\left(
			\begin{array}{c}
				-(x\partial_x)^2+nx\partial_x-(n-1) \\
				-(x\partial_x)^2+nx\partial_x
			\end{array}
			\right)+o(1),
		\end{equation}
		so we may take  $I_{\prescript{g}{}{\Delta}^1}=(1,n-1)$. 
		Since $H^1(\overline{\mathbb H^{n+1}},\mathbb S^n)=\{0\}$, it follows from Theorem \ref{mazz:top} that 
		$\mathcal N(\prescript{g}{}{\Delta}^1)$ is invertible and we conclude more generally the following useful result.
		
		\begin{theorem}\label{usula:inv}\cite{usula2021yang}
			If $(M^{n+1},g)$ is ALH with $n\geq 3$ then 
			\begin{equation}\label{usula:inv:t}
				\prescript{g}{}{\Delta}^1_\omega:x^\delta C^{k,\alpha}_0(\Lambda^1\prescript{0}{}{T^*M}\otimes \mathfrak l)\to x^\delta C^{k,\alpha}_0(\Lambda^1\prescript{0}{}{T^*M}\otimes \mathfrak l), \quad \delta\in(1,n-1),
			\end{equation}
			the Hodge Laplacian twisted by a genuine ``$0$-connection'' $\omega$ on a principal bundle $\overline P\to\overline M$, is Fredholm.
		\end{theorem}
		
		\begin{proof}
			As explained in the proof of \cite[Lemma 21]{usula2021yang}, and comparing with Remark \ref{rest:ext}, where it is explained what it is meant by $\omega$ being genuine, $\mathcal N(\prescript{g}{}{\Delta}^1_\omega)$ is the direct sum of $(\dim\mathfrak l)^2$ copies of $\mathcal N(\prescript{g}{}{\Delta}_1)$, so Theorem \ref{th:usef:geo} applies. The assumption $n\geq 3$ is needed here not only to use Theorem \ref{mazz:top} (with $k=1$) but also to ensure that $n/2\in I_{\prescript{g}{}{\Delta}^1}=(1,n-1)$.
		\end{proof}	
		
		We have seen in the discussion preceding Theorem \ref{th:usef:geo} that the scalar Laplacian $\prescript{g}{}{\Delta}=\prescript{g}{}{\Delta}^0$ is always invertible for $\delta\in(0,n)$. It turns out that the same reasoning as in the proof of Theorem \ref{usula:inv}, based on Remark \ref{rest:ext}, when combined with a simple application of the maximum principle, yields the same conclusion in the twisted case.

		\begin{theorem}\cite{usula2021yang}\label{lap:twist}
			If $(M,g)$ is ALH and $\omega$ is as in Theirem \ref{usula:inv} then 
			\[
			\prescript{g}{}{\Delta}^0_\omega=	\prescript{g}{}{d}^*_\omega d_\omega :x^\delta C^{k,\alpha}_0(\prescript{0}{}{T^*M}\otimes \mathfrak l)\to x^\delta C^{k,\alpha}_0(\prescript{0}{}{T^*M}\otimes \mathfrak l), \quad \delta\in(0,n),
			\] 
			is invertible.
			In particular, 
			\[
			\prescript{g}{}{d}^*_\omega:x^\delta C^{k,\alpha}_0(\Lambda^1\prescript{0}{}{T^*M}\otimes \mathfrak l)\to x^\delta C^{k,\alpha}_0(\prescript{0}{}{T^*M}\otimes \mathfrak l), \quad \delta\in(0,n),
			\]
			is surjective. 
		\end{theorem}

		We close this section by indicating how the theory above may be used to establish a mapping  result which, together with Theorem \ref{lap:twist}, is crucial in implementing the Bianchi-Coulomb gauge; compare with \cite[Lemma I.1.4]{biquard2000metriques}.
		
		\begin{theorem}\label{bian:surj}
			If $(M,g)$ is $PE$  then the map 
			\begin{equation}\label{bian:inv:eq}
				\prescript{g}{}{\mathcal B}\prescript{g}{}{\delta}^*:x^\delta C_0^{k,\alpha}(\Lambda^1\prescript{0}{}{T^*M})\to x^\delta C_0^{k-2,\alpha}(\Lambda^1\prescript{0}{}{T^*M}), \quad 0<\delta<n,
			\end{equation}
			is invertible. In particular, 
			the Bianchi map 
			\begin{equation}\label{bian:surj:eq}
				\prescript{g}{}{\mathcal B}:x^\delta C_0^{k,\alpha}(S^2\prescript{0}{}{T^*M})\to x^\delta C_0^{k-2,\alpha}(\Lambda^1\prescript{0}{}{T^*M}), \quad 0<\delta<n,
			\end{equation}
			is surjective. Moreover, the same result holds true for any ALH metric $g'$ on $M$ which is sufficiently close to $g$. 
			\begin{proof}
				A well-known computation shows that 
				\begin{equation}\label{ricci:boch}
					2\prescript{g}{}{\mathcal B}\prescript{g}{}{\delta}=\nabla^*\nabla-\prescript{g}{}{\rm Ric},
				\end{equation}
				a geometric, formally self-adjoint operator to which Theorem \ref{th:usef:geo} applies. 
				To check this, use Bochner formula to get 
				\begin{eqnarray*}
					2\prescript{g}{}{\mathcal B}\prescript{g}{}{\delta}
					& = & \prescript{g}{}{\Delta^1}-2\prescript{g}{}{\rm Ric}\\
					& = &  \prescript{g}{}{\Delta^1}+2n +o(1),
				\end{eqnarray*}
				and combining this with (\ref{ind:delta1}) we see that 
				\[
				I_{	2\prescript{g}{}{\mathcal B}\prescript{g}{}{\delta}}=\left(\frac{n}{2}-\frac{1}{2}\sqrt{n^2+8n},\frac{n}{2}+\frac{1}{2}\sqrt{n^2+8n}\right)\supset (0,n). 
				\]
				Taking into account that $\prescript{g}{}{\rm Ric}=-ng<0$ everywhere, the proof is completed for $g$ a PE metric along the lines above by means of the standard vanishing argument applied to (\ref{ricci:boch}). 
				As for the last assertion, note that if $g'$ is close to $g$ then  the $0$-calculus provides a parametrix $\prescript{g'}{}{H}$ for $\prescript{g'}{}{\mathcal B}\prescript{g'}{}{\delta}^*$ in the sense that $\prescript{g'}{}{H}\prescript{g'}{}{\mathcal B}\prescript{g'}{}{\delta}^*=\prescript{g'}{}{\mathcal B}\prescript{g'}{}{\delta}^*\prescript{g'}{}{H}={\rm Id}-\prescript{g'}{}{O}$, where $\prescript{g'}{}{O}$ is the residual term and with everything in sight depending continuously on the metric. Since we already know that $\prescript{g}{}{O}=0$, we may take $\prescript{g'}{}{O}$ small enough so as to make ${\rm Id}-\prescript{g'}{}{O}$ invertible. Thus, $\prescript{g'}{}{\mathcal B}\prescript{g'}{}{\delta}^*$ is invertible as well; compare with \cite[Theorem 44]{biquard2011nonlinear}. 
			\end{proof}
		\end{theorem}

		\section{Gauge fixing}\label{fix:gauge}

		As already mentioned, the gauge action in the configuration space $\mathcal M_M\times\mathcal C_P$ leaves invariant the EYM action and hence prevents the EYM field equations (\ref{eym:field:rew}) from being elliptic. The standard way to overcome  this difficulty is to introduce a gauge slice along which ellipticity is restored. Here we accomplish this by means of the so-called {\em Bianchi-Coulomb gauge}, which combines ideas and procedures from \cite{biquard2000metriques,mazzeo2006maskit,usula2021yang}. 
		
		In the following we take $k\geq 3$ and $1<\delta<2$. 
		Also, if we fix a background ALH manifold $(M^{n+1},g_0)$, $n\geq 3$, endowed with  a special defining function $x$ in a collar neighborhood $\mathcal U$ of its conformal boundary $(Y,\partial_\infty g_0)$.
		Given a smooth configuration $(g',\omega')$ on $Y=\partial \overline M$, where $\omega'$ is assumed to be genuine as in Remark \ref{rest:ext}, we define an affine Banach manifold of configurations $\mathscr F^{(k),\alpha}$ restricted to $Y$ by  
		\[
		\mathscr F^{(k),\alpha}=\left(g'+C^{k+1,\alpha}(S^2T^*Y),\omega'+C^k(T^*Y\otimes {\rm Ad}\,{\mathfrak l})\right). 
		\]
		Also, we note the existence of {\em bounded} extension maps
		\[
		e:C^{k+1,\alpha}(S^2T^*Y)\to C^{k+1,\alpha}(S^2T^*\overline{M})
		\]
		and 
		\[
		\mathfrak e:C^k(T^*Y\otimes {\rm Ad}\,{\mathfrak l})\to C^k(T^*{\overline M}\otimes {\rm Ad}\, {\mathfrak l}), 
		\]
		which are essentially defined by composing parallel translation of the original object along the flow lines of $\prescript{x^2g_0}{}{\nabla}x$ with multiplication by a cutoff function that equals $1$ along $Y$ and has support contained in $\mathcal U$; see \cite{graham1991einstein,lee2006fredholm,usula2021yang} for the detailed constructions of these extensions.		 
		We use this to define a Banach manifold $\mathscr I_\delta^{(k),\alpha}$ whose elements restrict to elements of $\mathscr F^{(k),\alpha}$ along $Y$. Precisely, given a background configuration $(g_0,\omega_0)$ as above, we declare that $(g,\omega)\in \mathscr I_\delta^{(k),\alpha}$ if it takes the form
		\begin{equation}\label{decomp:data}
			(g,\omega)=(g_0,\omega_0)+e_x(\Gamma),\mathfrak e(\gamma))+(A,a), \quad e_x(\Gamma)=x^{-2}e(\Gamma),
		\end{equation}
		where 
		\[
		(\Gamma,\gamma)\in C^{k+1,\alpha}(S^2T^* Y)\times C^{k,\alpha}(T^*Y\otimes {{\rm Ad}\, \mathfrak l})
		\]
		and 
		\[(A,a)\in x^\delta C^{k+1,\alpha}_0(S^2T^*\overline M)\times x^\delta C^{k,\alpha}_0
		(T^*{\overline M}\otimes \overline{{{\rm Ad}\,\mathfrak l}}).
		\]
		The embedding theorems mentioned in Remark \ref{emb:theo} easily imply that this is a Banach manifold modeled on 
		\[
		\mathscr B:=C^{k+1,\alpha}(S^2T^* Y)\times  x^\delta C^{k+1,\alpha}_0(S^2T^*\overline M)\times C^{k,\alpha}(T^*Y\otimes{{\rm Ad}\, \mathfrak l})\times  x^\delta C^{k,\alpha}_0
		(T^*{\overline M}\otimes\overline{{{\rm Ad}\,\mathfrak l}}).
		\]
		Moreover, 
		\begin{equation}\label{bd:rest}
			\partial_\infty g=[x^2g_0|_Y+\Gamma],\quad 
			\omega\big|_Y=\omega_0\big|_Y+\gamma. 
		\end{equation}
		In particular, $(M,g)$ is  ALH. 
		
		Finally, we define a Banach Lie group $\mathscr G^{(k),\alpha}_\delta$ of gauge transformations formed by elements of $\mathscr G={\rm Diff}\,\overline M\times \mathpzc L_P$ which act on $\mathscr I_\delta^{(k),\alpha}$ commuting with the operation of passing to the boundary data as in (\ref{bd:rest}). It is formed by those continuous elements of $\mathscr G$ contained in 
		\[
		{\mathsf{Exp}}\left(x^\delta C_0^{k+2,\alpha}(T\overline M)\right)\times \left({\rm Id}+ x^\delta  C_0^{k+1}(T^*\overline M\otimes\overline{{\rm Ad}\,\mathfrak l})\right).
		\]
		Again using the embedding theorems, it is easy to check that  $\mathscr G^{(k),\alpha}_\delta$ has all the properties mentioned above.
		
		We are now ready to fix the gauge by means of the Bianchi-Coulomb recipe. We define $\mathscr S_{(g_0,\omega_0)}$ to be space of all configurations $(g,\omega)$ as in (\ref{decomp:data}) with $(g,\omega)\in \mathscr I_\delta^{(k),\alpha}$ and
		satisfying the {\em Bianchi-Coulomb gauge}
		\[
		\left\{
		\begin{array}{l}
			\prescript{g_0+e_x(\Gamma)}{}{\mathcal B}A=0 \\
			\prescript{g_0+e_x(\Gamma)}{}{d^*_{\omega_0+\mathfrak e(\gamma)}}a=0
		\end{array}
		\right.
		\]
		
		\begin{proposition}\label{slice:man}
			If $(M,g_0)$ is PE then, near $(g_0,\omega_0)$, $\mathscr S_{(g_0,\omega_0)}$ is a smooth Banach submanifold of $\mathscr I_\delta^{(k),\alpha}$ with 
			\begin{eqnarray*}
				T_{(g_0,\omega_0)}\mathscr S_{(g_0,\omega_0)}
				& = & 
				C^{k+1,\alpha}(S^2T^* Y)\times {\ker} \prescript{g_0}{}{\mathcal B}\big|_{x^\delta C^{k+1,\alpha}_0(S^2T^*M)}\times \\
				& & 	\quad \times C^{k,\alpha}(T^*Y\otimes{{\rm Ad}\, \mathfrak l})\times  {\ker}\prescript{g_0}{}{d}^*_\omega\big|_{x^\delta C^{k,\alpha}_0
					(T^*{\overline M}\otimes\overline{{{\rm Ad}\,\mathfrak l}})}.
			\end{eqnarray*}
		\end{proposition}
		
		\begin{proof}
			We have that $\mathscr S_{(g_0,\omega_0)}=\Xi^{-1}(0)$, where the smooth map
			\[
			\Xi: \mathscr I_\delta^{(k),\alpha}\to x^\delta C_0^{k,\alpha}(S^2T^*M)\times x^\delta C_0^{k-1,\alpha}(T^*M\big|_{{\rm Ad}\,\mathfrak l})
			\]
			is given by 
			\[ \Xi((g_0,\omega_0)+(e_x(\Gamma),\mathfrak e(\gamma))+(A,a))=(\prescript{g_0+e_x(\Gamma)}{}{\mathcal B}A,\prescript{g_0+e_x(\Gamma)}{}{d^*_{\omega_0+\mathfrak e(\gamma)}}a).
			\]
			Clearly, its differential at $(g_0,\omega_0)$,
			\[
			D_{(g_0,\omega_0)}\Xi:\mathscr B \to  x^\delta C_0^{k,\alpha}(S^2T^*M)\times x^\delta C_0^{k-1,\alpha}(T^*M\big|_{{\rm Ad}\,\mathfrak l})
			\]
			is 
			\[
			\left(D_{(g_0,\omega_0)}
			\Xi\right)(e_x(\Gamma)+A,\mathfrak e(\gamma)+a)= (\prescript{g_0}{}{\mathcal B}A,\prescript{g_0}{}{d^*_{\omega_0}}a).
			\]
			By Theorems \ref{lap:twist} and \ref{bian:surj}, this is surjective and the result follows from the Implicit Function Theorem. 
		\end{proof}
		
		That $\mathscr S_{(g_0,\omega_0)}$ is a {\em bona fide} local slice for the gauge action of $\mathscr G_\delta^{(k),\alpha}$ on the space of configurations  around a suitable $(g_0,\omega_0)$ is confirmed by the next result.
		
		\begin{proposition}\label{gauge:bona}
			If $(M,g_0)$ is PE then there exist neighborhoods
			\begin{itemize}
				\item $\mathscr U_{(g_0,\omega_0)}$ of $(g_0,\omega_0)$ in $\mathscr S_{(g_0,\omega_0)}$;
				\item $\mathscr V_{\rm Id}$ of the identity in $\mathscr G^{(k),\alpha}_\delta$;
				\item $\mathscr W_{(g_0,\omega_0)}$ of $(g_0,\omega_0)$ in $\mathscr I^{(k),\alpha}_\delta$
			\end{itemize}
			such that for any $(g,\omega)\in \mathscr W_{(g_0,\omega_0)}$ there exists a unique $\Phi\in \mathscr V_{{\rm Id}}$ such $(g,\omega)\cdot\Phi\in \mathscr U_{(g_0,\omega_0)}$. 
		\end{proposition}
		
		\begin{proof}
			Restrict the gauge action to obtain a map 
			\[
			\mathscr C:\mathscr S_{(g_0,\omega_0)}\times \mathscr G^{(k),\alpha}_\delta\to \mathscr I^{(k),\alpha}_\delta,
			\]
			so that
			\[
			\left(D_{((g_0,\omega_0),{\rm Id})}\mathscr C\right)((A,U),(a,u))=\left((\Gamma,2\prescript{g_0}{}{\delta}^*U+A),(\gamma,d_{\omega_0} u+a)\right),
			\]
			which maps
			\[
			T_{(g_0,\omega_0)}\mathscr S_{(g_0,\omega_0)}\times T_{{\rm Id}}\mathscr G^{(k),\alpha}_\delta \to 
			x^\delta C^{k+1,\alpha}_0(S^2T^*M)\times   x^\delta C^{k,\alpha}_0
			(T^*{\overline M}\otimes\overline{{{\rm Ad}\,\mathfrak l}}).
			\]
			Here we used (\ref{exp:con}), (\ref{exp:met}) and Lemma \ref{div:adj}. 
			We will check that this derivative map is invertible, which completes the proof by the Inverse Function Theorem. Let us first  see what happens with the ``left side'' of the map, which only involves metric information. To check injectivity, assume that $2\prescript{g_0}{}{\delta}^*U+A=0$. From the description of $	T_{(g_0,\omega_0)}\mathscr S_{(g_0,\omega_0)}$ in Proposition \ref{slice:man} we have $\prescript{g_0}{}{\mathcal B}A=0$, which gives $\prescript{g_0}{}{\mathcal B}\prescript{g_0}{}{\delta}^*U=0$. Thus, $U=0$ by Theorem \ref{bian:surj} and hence $A=0$ as well. As for the surjectivity, if $B\in  x^\delta C_0^{k-2,\alpha}(S^2\prescript{0}{}{T^*M})$ is  given we must find $(A,U)$ in the domain which is mapped onto $B$. First solve $2\prescript{g_0}{}{\mathcal B}\prescript{g_0}{}{\delta}^*U=\prescript{g_0}{}{\mathcal B}B$ for $U\in x^\delta C_0^{k,\alpha}(\Lambda^1\prescript{0}{}{T^*M})$ by means of Theorem \ref{bian:surj}. Now set $A:=2\prescript{g_0}{}{\delta}^*U-B$,
			which satisfies $\prescript{g_0}{}{\mathcal B}A=0$, as desired. As for the ``right'' side of the map involving the connection, the proof is entirely similar, uses Theorem \ref{lap:twist} and already appears in \cite[Proposition 33]{usula2021yang}.
		\end{proof}
		
		\section{The main result and its proof}\label{main:res}
		
		Here we retain the terminology, notation and requirements of the previous section. In particular, we take $n\geq 3$, $k\geq 3$ and $1<\delta<2$. We now single out the class of EYM fields to which the perturbative scheme will be applied.
		
		\begin{definition}\label{nondeg:conf}
			We say that a EYM field $(g_0,\omega_0)$ is {\em trivial} if $g_0$ is PE and $\omega_0$ is flat ($\Omega_{\omega_0}=0$). Moreover, a trivial $(g_0,\omega_0)$ is {\em non-degenerate} if 
			\begin{itemize}
				\item $g_0$ is non-degenerate as  a PE metric in the sense that
				$\ker \prescript{g_0}{}{\Delta}_{(n)}\big|_{L^2}$ is trivial;
				\item $\omega_0$ is non-degenerate as a flat connection in the sense that $\ker \prescript{g_0}{}{\Delta}^1_{\omega_0}\big|_{L^2}$ is trivial. 
			\end{itemize}
		\end{definition}
		
		\begin{remark}\label{non:flat}
			The second item above corresponds to a special case of the more general definition in \cite{usula2021yang}, where $\omega_0$ is not necessarily flat and $\prescript{g_0}{}{\Delta}_{\omega_0}^1$  gets replaced by $ \prescript{g_0}{}{\Delta}_{\omega_0}^1+\cdot^\sharp \iprod\Omega_{\omega_0}$. We note, however, that this notion is only effectively used in \cite{usula2021yang} in its Corollary 39, which requires flatness (indeed, triviality). As already observed there, this relates to the difficulty of finding {\em non-trivial} non-degenerate Yang-Mills connections to which the perturbative scheme may be applied. A similar difficulty also occurs here, which justifies Definition \ref{nondeg:conf}. 
		\end{remark}
		
		We may now state our main result. 
		
		\begin{theorem}\label{main}  If $(g_0,\omega_0)$ is a non-degenerate trivial EYM as above (with $n\geq 3$) then for any  sufficiently small pair $(\Gamma,\gamma)\in C^{k+1}(S^2T^*Y)\times C^k(T^*Y\otimes {\rm Ad}\,\mathfrak l)$  there exists a EYM field $(g,\omega)$ such that $(\partial_\infty g,\omega\big|_Y)=([x^2g_0\big|_Y+\Gamma],\omega_0\big|_Y+\gamma)$. Morever, $(g,\omega)$ is unique up to the action of the gauge group fixing the boundary data.   
		\end{theorem}
		
		To prove this we fix $(g_0,\omega_0)$ as in the theorem and consider the nonlinear, smooth map 
		\[
		\left(
		\begin{array}{c}
			g\\
			\omega
		\end{array}
		\right)
		=
		\left(
		\begin{array}{c}
			g_0+e_x(\Gamma)+A\\
			\omega_0+\mathfrak e(\gamma)+a
		\end{array}
		\right)
		\stackrel{\mathscr P}{\mapsto}
		\left(
		\begin{array}{c}
			\mathscr Q(g,\omega)\\
			\mathscr R(g,\omega)
		\end{array}
		\right)	,
		\]
		where we make use of the decomposition in (\ref{decomp:data}) and
		\begin{equation}\label{math:P}
			\left(
			\begin{array}{c}
				\mathscr Q(g,\omega)\\
				\mathscr R(g,\omega)
			\end{array}
			\right)	
			=\left(
			\begin{array}{c}
				\prescript{g}{}{\rm Ric}+ng-\widetilde K_{g,\omega}+\prescript{g}{}{\delta}^*\prescript{g_0+e_x(\Gamma)}{}{\mathcal B}A\\
				\prescript{g}{}{d}^*_\omega\Omega_\omega+d_\omega\prescript{g_0+e_x(\Gamma)}{}{d}^*_{\omega_0+\mathfrak e(\gamma)}A
			\end{array}
			\right).
		\end{equation}
		If $\Gamma$, $A$, $\gamma$ and $a$ are taken sufficiently small,  we may assume that this is defined in a small neighborhood $\mathscr W$ of $(g_0,\omega_0)$ in  $\mathscr W_{(g_0,\omega_0)}$; see Proposition \ref{gauge:bona}.		
		Notice that $\mathscr P(g_0,\omega_0)=(0,0)$. 
		Clearly, if $(g,\omega)$ is EYM and lies in the slice $\mathscr S_{(g_o,\omega_0)}$ then $(g,\omega)\in \mathscr P^{-1}(0,0)$. 
		Conversely, assume that $(g,\omega)\in \mathscr P^{-1}(0,0)$, so that
		\begin{equation}\label{conv:gauge}
			\left\{
			\begin{array}{c}
				\prescript{g}{}{\rm Ric}+ng-\widetilde K_{g,\omega}+\prescript{g}{}{\delta}^*\prescript{g_0+e_x(\Gamma)}{}{\mathcal B}A=0\\
				\prescript{g}{}{d}^*_\omega\Omega_\omega+d_\omega\prescript{g_0+e_x(\Gamma)}{}{d}^*_{\omega_0+\mathfrak e(\gamma)}a=0	
			\end{array}
			\right.
		\end{equation}
		Applying $\prescript{g}{}{\mathcal B}$ to the first equation above and using (\ref{bian:id:2}) and Lemma \ref{usef:lem}  we see that $\prescript{g}{}{\mathcal B}\prescript{g}{}{\delta}^*\prescript{g_0+e_x(\Gamma)}{}{\mathcal B}A=0$. Since $g-g_0=e_x(\Gamma)+A$ is small we may apply Theorem \ref{bian:surj} to conclude that $\prescript{g_0+e_x(\Gamma)}{}{\mathcal B}A=0$.  A similar argument using the second equation in (\ref{conv:gauge}), the identity $\prescript{g}{}{d}^*_\omega	\prescript{g}{}{d}^*_\omega\Omega_\omega=0$ in (\ref{off:shell:con}), which holds {\em off-shell}, and Theorem \ref{lap:twist} allows us to conclude that $\prescript{g_0+e_x(\Gamma)}{}{d}^*_{\omega_0+\mathfrak e(\gamma)}a=0$; compare with \cite[Proposition 36]{usula2021yang}. Thus, $(g,\omega)\in \mathscr S_{(g_o,\omega_0)}$ and going back to (\ref{conv:gauge}) we see that $(g,\omega)$ is EYM as well. 
		The upshot here is that proving Theorem \ref{main} is equivalent to solving the gauged boundary value problem
		
		\begin{equation}\label{dirich}
			\left\{
			\begin{array}{l}
				\mathscr P_{(g_0,\omega_0)}(g,\omega)=0, \,\,(g,\omega)\in \mathscr I_{\delta}^{k,\alpha}\\
				(g,\omega)\in \mathscr S_{(g_0,\omega)}\\
				(\partial_\infty g,\omega\big|_Y)=([x^2g_0|_Y+\Gamma],\omega_0\big|_Y+\gamma)	
			\end{array}
			\right.
		\end{equation}
		for $(g,\omega)$ close to $(g_0,\omega_0)$. 
		
		\begin{lemma}\label{differ:cal}	
			If $(g_0,\omega_0)$ is trivial then there holds
			\[
			D\mathscr P_{(g_0,\omega_0)}
			\left(
			\begin{array}{c}
				(e_x(\Gamma),A)\\
				(\mathfrak e(\gamma),a)
			\end{array}
			\right)=
			\left(
			\begin{array}{cc}
				\frac{1}{2}\left(\prescript{g_0}{}{\mathscr L}(e_x(\Gamma))+\prescript{g_0}{}{\Delta}_nA\right)
				& 0\\
				0 & 
				\prescript{g_0}{}{d}^*_{\omega} d_{\omega}\mathfrak e(\gamma)+\prescript{g_0}{}{\Delta}^1_{\omega_0}a
			\end{array}
			\right),
			\]
			where $\prescript{g_0}{}{\mathscr L}$ is the operator appearing in (\ref{exp:met:2}). 
			
		\end{lemma}	
		
		\begin{proof}
			We write symbolically
			\begin{equation}\label{symb:part}
				D\mathscr P_{(g_0,\omega_0)}=
				\left(
				\begin{array}{cc}
					\partial_g\mathscr Q
					& \partial_\omega \mathscr Q\\
					\partial_g \mathscr R & \partial_\omega \mathscr R  
				\end{array}
				\right),
			\end{equation}
			where for convenience we omit $(g_0,\omega_0)$ in the partial derivatives on the right-hand side. We now observe that, 
			due to Remark \ref{K:quad} and the flatness of $\omega_0$, the terms $-\widetilde K_{g,\omega}$ 
			and  $\prescript{g}{}{d}^*_{\omega}\Omega_\omega$ in (\ref{math:P}) do not contribute to $\partial_\omega \mathscr Q$ and to $\partial_{g}\mathscr R$, respectively.
			Thus, the off-diagonal terms in (\ref{symb:part}) vanish. By the same reason, the term $-\widetilde K_{g,\omega}$ does not contribute to $\partial_g \mathscr Q$, so if we use (\ref{exp:met:2}) with $g=g_0$ and $h=e_x(\Gamma)+A$ we obtain
			\begin{eqnarray*}
				\partial_g\mathscr Q
				& = & 
				\frac{1}{2}\left(\prescript{g_0}{}{\mathscr L}(e_x(\Gamma))+\prescript{g_0}{}{\mathscr L}A\right)+ \prescript{g_0}{}{\delta}^*\prescript{g_0}{}{\mathcal B}A\\
				& = & \frac{1}{2}\left(\prescript{g_0}{}{\mathscr L}(e_x(\Gamma))+\prescript{g_0}{}{\Delta}_{(n)}A\right)-\frac{1}{2}\mathbb L_{{(\prescript{g_0}{}{\mathcal B}}A)^\sharp}+ \prescript{g_0}{}{\delta}^*\prescript{g_0}{}{\mathcal B}A\\
				& = & \frac{1}{2}\left(\prescript{g_0}{}{\mathscr L}(e_x(\Gamma))+\prescript{g_0}{}{\Delta}_{(n)}A\right),
			\end{eqnarray*}
			where we used Lemma \ref{div:adj} in the last step. Finally, the computation of the remaining entry in the right-hand side of (\ref{symb:part}) follows as in \cite{usula2021yang}. 
		\end{proof}
		
		The next result completes the proof of Theorem \ref{main}. 
		
		\begin{proposition}\label{comp:main}
			If $(g_0,\omega_0)$ is a trivial and non-degenerate EYM field then there exist neighborhoods 
			\begin{itemize}
				\item $\mathsf V$ of $(0,0)$ in $C^{k+1,\alpha}(S^2T^*Y)\times C^k(T^*Y\otimes {\mathfrak l})$,
				\item $\mathsf W$ of $(0,0)$ in $ x^\delta C^{k+1,\alpha}_0(S^2T^*\overline M)\times x^\delta C^{k,\alpha}_0
				(T^*{\overline M}\otimes \overline{{{\rm Ad}\,\mathfrak l}})$
			\end{itemize}
			and a unique smooth map $\mathsf a:\mathsf V\to\mathsf W$ such $\mathsf a(0,0)=(0,0)$ and for every $(\Gamma,\gamma)\in \mathsf U$, $(g_0,\omega_0)+(e_x(\Gamma),\mathfrak e(\gamma))+\mathsf a(\Gamma,\gamma)$ is a solution of (\ref{dirich}). 
		\end{proposition}
		
		\begin{proof}
			By Lemma \ref{differ:cal},
			\[
			D\mathscr P_{(g_0,\omega_0)}
			\left(
			\begin{array}{c}
				(0,A)\\
				(0,a)
			\end{array}
			\right)=
			\left(
			\begin{array}{cc}
				\frac{1}{2}\prescript{g_0}{}{\Delta}_{(n)}A
				& 0\\
				0 & 
				\prescript{g_0}{}{\Delta}^1_{\omega_0}a,
			\end{array}
			\right),
			\]  
			and the result follows from the Implicit Function Theorem in view of Corollary \ref{cor:rel} and Theorem \ref{usula:inv}.
		\end{proof}
		
		The next result is a consequence of Theorem \ref{main} worth mentioning and corresponds to \cite[Corollary 39]{usula2021yang}, which by its turn justifies an expectation due to Witten \cite{witten1998antide}.  
		
		\begin{theorem}\label{witten}
			Let $(M,g_0)$ be a PE manifold satisfying  $H^1(\overline M,Y)=\{0\}$ and with $g_0$  non-degenarate. Also, let $\omega_0$ be the canonical flat connection  on a trivial principal bundle $P$ over $M$. Then the conclusion of Theorem \ref{main} holds. In particular, we may take $(M,g_0)=(\mathbb H^{n+1},g_{\mathbb H})$. 
		\end{theorem}
		
		\begin{proof}
			By Theorem \ref{mazz:top} with $k=1$, the cohomological condition on  $(\overline M,Y)$ implies that $\omega_0$ is non-degenerate. The last assertion follows from Corollary \ref{cor:rel}.  	
		\end{proof}
		
		\section{Further results and questions}\label{further}
		
		There are at least two situations to which our main deformation result (Theorem \ref{main}) may be extended in a rather straightforward manner:
		\begin{itemize}
			\item As in \cite{mazzeo2006maskit} we may conceive a  scenario (perhaps more realistic in the context of certain applications such as orbifold degeneration \cite{girao2010orbifold}) in which the background PE manifold develops finitely many isolated conical singularities. Thus, it is natural to require that the deformed EYM configurations induced by slight perturbations of the boundary data should also display a similar behavior around these interior singular points. Since the mapping theory of generalized Laplacians around conical singularities is well understood \cite{lesch1997differential,mazzeo1991elliptic}, the appropriate version of Theorem \ref{main} in this slightly more general setting may be readily obtained without much extra effort.
			\item  We may replace the purely gravitational sector of the action (\ref{eym:action}), $\prescript{g}{}{\rm R}+n(n-1)$, by $\prescript{g}{}{\rm R}_{2k}+c_{k,n}$, where $\prescript{g}{}{\rm R}_{2k}$, $2\leq k<(n+1)/2$, is the $2k$-{\em Gauss-Bonnet curvature} of $g$ \cite{labbi2008variational} and the positive constant $c_{k,n}$ is suitably chosen. Upon extremization and a little algebraic manipulation, the corresponding field equations assume a form quite similar to (\ref{eym:field:rew}), in which $\prescript{g}{}{\rm Ric}+ng$ gets replaced by $\prescript{g}{}{\rm Ric}_{2k}+d_{k,n}g$, where $\prescript{g}{}{\rm Ric}_{2k}$ is the $2k$-{\rm Ricci tensor} of $g$ \cite{labbi2008variational} and $d_{k,n}>0$ is determined by the requirement that $\prescript{g_{\mathbb H}}{}{\rm Ric}_{2k}+d_{k,n}g_{\mathbb H}=0$. In general, the linearization  of this ``Lovelock map'' $g\mapsto \prescript{g}{}{\rm Ric}_{2k}+d_{k,n}g$ around an arbitrary metric is too complicated to be of any use (in fact, even after discarding the Lie derivative term coming from diffeomorphism invariance, it might remain non-elliptic as its principal symbol in general depends on the underlying curvature). However, as observed in \cite{delima2010deformations,caula2013deformation}, its Bianchi-gauged linearization around $g_{\mathbb H}$ equals $\prescript{g_{\mathbb H}}{}{\Delta}_{(n)}$ up to a non-zero multiplicative constant. Thus,  essentially the same argument as the one leading to Theorem \ref{main} implies the existence of nearby ``Lovelock-Yang-Mills'' fields around the trivial configuration $(g_{\mathbb H},\omega_0)$, where $\omega_0$ is any trivial connection, by slightly perturbing the corresponding boundary data and solving the associated Dirichlet problem at infinity. This may be viewed as a natural extension of \cite[Theorem 3.1]{albin2020poincare}, which establishes the existence of Poincar\'e-Lovelock metrics on the unit ball by similar methods.      
		\end{itemize} 
		
		We now briefly discuss a few, perhaps more sophisticated, lines of research which potentially may be incorporated to our basic Einstein-Yang-Mills setting. From an exclusively perturbative viewpoint, it may be worthwhile to consider the problem of extending Theorem \ref{main} to background Einstein metrics with a more complicated structure at infinity, as  those modeled on (possibly higher rank) symmetric spaces and treated in \cite{biquard2000metriques,biquard2006parabolic,biquard2011nonlinear,bahuaud2020geometrically}. We may also consider the prospect of gluing EYM configurations  along  the boundary data in  the line of the ``Maskit combination'' performed in \cite{mazzeo2006maskit}. Finally, in the spirit of \cite[Chapter 5]{bleecker2005gauge},
		we may enrich the action (\ref{eym:action}) by adding a density term involving covariant derivatives up to first order of a ``particle field'' $\psi$, a section of a fixed vector bundle associated to some orthogonal representation of $L$. This additional degree of freedom turns the second field equation in (\ref{eym:fields}) in-homogeneous, with the appearance  of an extra term in its right-hand side, the current $J_\omega(\psi)\in \mathcal A^1_{\rm Ad}(P,\mathfrak l)$, obtained by varying the particle density with respect to $\omega$. Virtually any question studied or mentioned so far may be rephrased in this more general setting and certainly deserves further investigation.

		\bibliographystyle{alpha}
		\bibliography{EYM}		
	\end{document}